\documentclass[10pt,reqno]{amsart}
\usepackage{amsmath,amssymb}
\usepackage{color}
\usepackage{soul}
\usepackage[colorlinks=true,linkcolor=blue,citecolor=magenta]{hyperref}
\usepackage[margin=1in]{geometry}
\usepackage{hyperref}

\newcommand{\ainf}{a_{\infty}}
\newcommand{\binf}{b_{\infty}}
\newcommand{\sign}{\mathrm{sign}}

\makeatletter
\newcommand{\pushright}[1]{\ifmeasuring@#1\else\omit\hfill$\displaystyle#1$\fi\ignorespaces}
\newcommand{\pushleft}[1]{\ifmeasuring@#1\else\omit$\displaystyle#1$\hfill\fi\ignorespaces}
\makeatother
\title[Reaction-diffusion systems with nonlinear diffusion]{Global regularity and convergence to equilibrium of reaction-diffusion systems with nonlinear diffusion}
\author{Klemens Fellner, Evangelos Latos, Bao Quoc Tang}
\address{Klemens Fellner \hfill\break
	Institute of Mathematics and Scientific Computing, University of Graz, Heinrichstrasse 36, 8010 Graz, Austria}
\email{klemens.fellner@uni-graz.at}
\address{Evangelos Latos \hfill\break
	University of Mannheim, D-68131 Mannheim, Germany}
\email{evangelos.latos@math.uni-mannheim.de}
\address{Bao Quoc Tang \hfill\break
	Institute of Mathematics and Scientific Computing, University of Graz, Heinrichstrasse 36, 8010 Graz, Austria}
\email{quoc.tang@uni-graz.at} 
\begin{document}
\newtheorem{lemma}{Lemma}[section]
\newtheorem{remark}{Remark}[section]
\newtheorem{theorem}{Theorem}[section]	
\newcommand{\iTO}{\int_0^T\!\!\int_{\Omega}}	
\begin{abstract}
		We study the boundedness and convergence to equilibrium of weak solutions to reaction-diffusion systems with nonlinear diffusion. The nonlinear diffusion is of porous medium type and the nonlinear reaction terms are assumed to grow  polynomially and to dissipate (or conserve) the total mass. By utilising duality estimates, the dissipation of the total mass and the smoothing effect of the porous medium equation, we prove that if the exponents of the nonlinear diffusion terms are high enough, then weak solutions are bounded, locally H\"older continuous and their $L^{\infty}(\Omega)$-norm grows in time at most polynomially. 
		
In order to show convergence to equilibrium, we consider a specific class of nonlinear reaction-diffusion models, which describe a single reversible reaction with arbitrarily many chemical substances. By exploiting a generalised Logarithmic Sobolev Inequality, an indirect diffusion effect and the polynomial in time growth of the $L^{\infty}(\Omega)$-norm, we show an entropy entropy-production inequality which implies exponential convergence to equilibrium in $L^p(\Omega)$-norm, for any $1\leq p < \infty$, with explicit rates and constants.
\end{abstract}
\maketitle
\tableofcontents
\section{Introduction and Main results}
In this article, we study the boundedness and convergence to equilibrium of weak solutions to reaction-diffusion systems with nonlinear diffusion
	\begin{equation}\label{S}\tag{S}
	\begin{cases}
	\partial_t u_i - d_i\Delta(u_i^{m_i}) = f_i(u),  &\quad x\in\Omega, \quad\ \; t>0, \qquad i=1,\ldots, S,\\
	d_i\nabla (u_i^{m_i})\cdot \overrightarrow{n} = 0, &\quad x\in\partial\Omega, \quad t>0, \qquad i=1,\ldots, S,\\
	u_i(x,0) = u_{i,0}(x), &\quad x\in\Omega, \qquad\qquad\qquad\; i=1,\ldots, S,
	\end{cases}
	\end{equation}
with the unknown functions $u = (u_1, \ldots, u_S)$ and $u_i: \Omega \times \mathbb R_+ \mapsto \mathbb R$,  the positive diffusion coefficients $d_i >0$, the porous medium exponents $m_i >1$ and where $\Omega \subset \mathbb R^d$ denotes a bounded domain with sufficiently smooth boundary $\partial\Omega$ (e.g. $\partial\Omega$ is of class $C^{2+\epsilon}$ for some $\epsilon>0$) with outward unit normal $\overrightarrow{n}$ on $\partial\Omega$. 
Moreover, the conditions imposed on the nonlinear reaction terms $f_i(u)$ and the nonnegative initial data $u_{i,0}$ will be specified later. 
\medskip	
	
The first part of this paper considers weak solutions to system \eqref{S}. Our aim is to provide sufficient conditions on the porous medium exponents $m_i$ and on the nonlinearities $f_i(u)$,
under which weak solutions are indeed bounded in $L^{\infty}$ (and thus locally H\"older-continuous) for all times and grow at most polynomially in time. 
More precisely, we assume the following conditions on the nonlinearities:
	\begin{itemize}
		\item[(i)] The nonlinearities $f_i: \mathbb R^S \to \mathbb R$ are locally Lipschitz functions and satisfy
		\begin{equation}\label{G}\tag{G}
		|f_i(u)|\leq C (1+|u|^\nu),\quad\forall u=(u_1,\ldots,u_S)\in\mathbb{R}^S, \quad \forall i=1,\ldots, S,	
		\end{equation}
		where $\mathbb R\ni \nu\geq 1$ is the maximal growth exponent of the reaction terms.
		\item[(ii)] There exist positive constants $\lambda_1,\ldots,\lambda_S>0$ such that:
		\begin{equation}\label{M}\tag{M}
		\sum^S_{i=1}\lambda_if_i(u)\leq0,\qquad\forall u\in\mathbb{R}^S,	
		\end{equation}
		which formally implies the following mass dissipation law
		$$
		\frac{d}{d t}\int_\Omega\sum^S_{i=1}\lambda_iu_idx\leq0.
		$$
		\item[(iii)] The nonlinearities are assumed quasi-positive, that is for all $i=1,\ldots,S,$ holds
		\begin{equation}\label{P}\tag{P}
		f(u_1,\ldots,u_{i-1},0,u_{i+1},\ldots,u_S)\geq0,\qquad\forall u_1,\ldots,u_S\geq0.
		\end{equation}
		The quasi-positivity condition \eqref{P} ensures global nonnegativity of solutions subject to nonnegative initial data, see e.g. \cite{Pie10,LP17}.
	\end{itemize}	

%The following weak solutions to system \eqref{S} where recently proven in \cite{LP17}.	 
The existence of global weak solutions to \eqref{S} subject to homogeneous Dirichlet boundary conditions and under the assumptions \eqref{G}-\eqref{M}-\eqref{P} was recently obtained in \cite{LP17}. 
{The proof of the following Theorem \ref{thm:1.1} on the existence of weak solutions to \eqref{S} subject to Neumann boundary conditions uses similar arguments to \cite{LP17} and is postponed to Section \ref{app}.}
%. For the sake of completeness, we will give a proof in Appendix \ref{app}.
\begin{theorem}\label{thm:1.1}
	Assume the conditions \eqref{G}, \eqref{M} and \eqref{P} and consider nonnegative initial data $(u_{i,0}) \in L^2(\Omega)^S$. If 
	\begin{equation*}
		m_i > \max\{\nu - 1; 1\} \quad \text{ for all } \quad i=1\ldots S,
	\end{equation*}
	then, there exists a global weak nonnegative solution to system \eqref{S} in the sense that, for all $i=1,\ldots, S$, $u_i \in C([0,+\infty); L^1(\Omega))$, $u_i^{m_i}\in L^1(0,T;W^{1,1}(\Omega))$, $f_i(u)\in L^1(\Omega \times [0,T])$ and 
		\begin{equation*}
			-\int_{\Omega}\psi(0)u_{i,0}dx - \int_{0}^{T}\int_{\Omega}(u_i\partial_t\psi + d_iu_i^{m_i}\Delta\psi)dxdt = \int_0^T\int_{\Omega}\psi f_i(u)dxdt
		\end{equation*}
		for all test function $\psi \in C^{2,1}(\overline{\Omega}\times [0,T])$ with $\nabla \psi\cdot \overrightarrow{n} = 0$ on $\partial\Omega\times (0,T)$ and $\psi(\cdot, T) = 0$.
		
Moreover, a solution $u = (u_1,\ldots, u_S)$ to \eqref{S} with \eqref{M} and \eqref{P} satisfy
\begin{equation*}%\label{duality}
	\|u_i\|_{L^{m_i+1}(Q_T)} \leq C \quad \text{ for all } \quad T>0 \quad \text{ and } \quad i = 1,\ldots, S,
\end{equation*}
where the constant $C$ depends on the $L^2$-norm of the initial data, the constants $\lambda_i$ in \eqref{M}, the diffusion coefficients $d_i>0$ and the domain $\Omega$.
	\end{theorem}
\begin{remark}
With a more careful analysis, it seems possible to generalise Theorem \ref{thm:1.1} and consider initial data $u_{i,0}\in L^1(\Omega)$. We refer the interested reader to \cite{PR16} for the case of systems with quadratic nonlinearities and $L^1$ initial data. 
\end{remark}

\medskip	
Given the weak solutions of Theorem \ref{thm:1.1}, our aim is to establish their boundedness and a polynomially in time growing $L^{\infty}$-estimate under stronger assumptions on the porous medium exponents $m_i$: First, we recall the 
a-priori estimate $u_i \in L^{m_i+1}(Q_T)$ %the condition $m_i > \nu - 1$, which is assumed for the global weak solutions 
of Theorem \ref{thm:1.1} %implies %by means of the duality estimate Lemma \ref{duality} and the mass dissipation condition \eqref{M} that $u_i \in L^{m_i+1}(Q_T)$, where $Q_T = \Omega\times (0,T)$. Hence, by 
and the growth condition \eqref{G} imply $f_i(u)\in L^{1+\epsilon}(Q_T)$ for some $\epsilon>0$, which also justifies the definition of weak solutions in Theorem \ref{thm:1.1}. In fact, the $L^{1+\epsilon}$ integrability guarantees uniform integrability of nonlinearities in a suitable approximating scheme (see the proof of Theorem \ref{thm:1.1} in the Section \ref{app}). 

Intuitively, Theorem \ref{thm:1.1} states that larger exponents $m_i$ yield higher integrability of the nonlinearities $f_i(u)$. Moreover, the functions $u_i$ solve a porous medium equation with the right hand side having higher integrability. Thus, by quantifying the smoothing effect from the porous medium equation, %and by using the mass dissipation property \eqref{M}, 
this allows to start a bootstrap argument, which eventually leads to boundedness of $u_i$ in $L^{\infty}$. In particular, it is of importance that our argument allows to show that the growth in time of the $L^{\infty}$-norms is at most polynomial. The first main result of this article is the following theorem. 

\begin{theorem}[Global bounded weak solutions]\label{T1.2}\hfill\\
Let $\Omega\subset\mathbb{R}^d$ be bounded with sufficiently smooth boundary. Let the initial data $0\le u_{i,0}\in L^\infty(\Omega)$, assume the conditions \eqref{G},\eqref{M} and \eqref{P} and 
{ 
$m_i > \max\{\nu - 1; 1\}$ for all $i=1\ldots S$ as required by Theorem \ref{thm:1.1}. 
Finally, {in dimensions $d\geq 3$, we additionally assume}
\begin{equation}\label{condstrong}
m_i>\nu - \frac{4}{d+2},\qquad\forall i=1\ldots S.
\end{equation}}

Then, any weak solution of \eqref{S} obtained in Theorem \ref{thm:1.1} is bounded in $L^{\infty}(\Omega)$ and grows in time at most polynomially in the sense that, for any $T>0$,
\[
\|u_i\|_{L^\infty(Q_T)}\leq C_T,\quad\forall i=1\ldots S
\]
where $C_T$ is a constant which depends at most polynomially on time. Consequently, these solutions are locally (in $Q_T$) H\"older continuous, see e.g. \cite{Vaz07}. 
%\textcolor{red}{I guess we should rather state a Theorem with H\"older continuous solutions with polynomially growing 
%$L^{\infty}$-norm.}

%\textcolor{red}{Uniqueness? Reference....I don't believe that this is easy...}
\end{theorem}
%\begin{remark}
%Note that in one or two space dimensions, condition \eqref{condstrong} on the nonlinear diffusion exponents in Theorem \ref{T1.2} is implied from the assumption in Theorem \ref{thm:1.1} for the existence of weak solutions. 
%Only in higher space dimension, i.e. $d\geq 3$, condition \eqref{condstrong} is indeed a stronger assumption.
%
%\textcolor{red}{Should we keep this remark?}
%\end{remark}
\begin{remark}[Weakened assumptions on mass dissipation and initial data]
	If one is only interested in the boundedness of solutions but not in the polynomial growth of the $L^{\infty}$-norm, then the mass dissipation condition \eqref{M} can in fact be weakened to
	\begin{equation*}
		\sum_{i=1}^{S}\lambda_i f_i(u) \leq C_1\sum_{i=1}^{S}|u_i| + C_2 \quad \text{ for all } u \in \mathbb R^S,
	\end{equation*}
	for some positive constants $C_1, C_2$.

Also the assumed initial regularity $u_{i,0}\in L^\infty(\Omega)$ is not optimal and could be relaxed 
to $L^{p}$ integrability for sufficiently large $p$ according to the details of the proof yet at the price of the readability 
of the Theorem.  
\end{remark}
\medskip

Theorem \ref{T1.2} contributes to the large literature on global existence and boundedness of solutions to reaction-diffusion systems, which nevertheless poses still many open questions due to the lack of a unified approach (maximum principles do not hold for general systems). The largest part of the available literature, however, considers the case of linear diffusion, i.e. $m_i = 1$ in system \eqref{S}. 
%has a massive literature concerning global well-posedness and boundedness of solutions, in particular for systems satisfying the conditions \eqref{P}-\eqref{M}. 
We refer the reader to the extensive review of Michel Pierre \cite{Pie10} and the references therein, in particular
\cite{
Ba94,
BeRe10,
BouPi00,
CaVa09,
DeFePiVo07,
HeLaVe98,
HMP87,
KaKi00,
Laam11,
Ma83,
Mor89,
Pi03,
PiSchmi97} % \textcolor{red}{more references}. 

The case of nonlinear diffusion, on the other hand, is much less investigated. Most of the existing results considered special systems with special structures, see e.g. \cite{Smo94,Leu09} . Up to the best of our knowledge,  system \eqref{S} under the general structural assumptions \eqref{G}-\eqref{M}-\eqref{P} was only studied very recently in  \cite{LP17}, where the authors showed the global existence of weak solutions. Therefore, the present paper serves as the first result to show the boundedness of weak solutions by assuming stronger conditions on porous medium exponents. 
Moreover, our proof allows to estimate explicitly the growth in time of the $L^{\infty}$-norm, which turns out to be essential in studying the large time behaviour of solutions in the following second part of the paper.

\medskip
The second main result of this paper proves exponential convergence to equilibrium for a class of 
reaction-diffusion systems with porous media diffusion of the form \eqref{S}, where 
the nonlinearities model the following reversible reaction with arbitrarily many chemical substances 
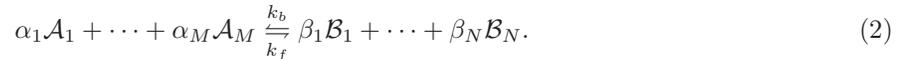
\begin{equation}\label{revrec}
\alpha_1\mathcal{A}_1+\cdots+\alpha_M\mathcal{A}_M
\underset{k_f}{\overset{k_b}{\leftrightharpoons}}
\beta_1\mathcal{B}_1+\cdots+\beta_N\mathcal{B}_N.
\end{equation}
Here $\alpha_i,\beta_i\in[1,+\infty)$ are the stoichiometric coefficients of the $M+N$ involved substances $\mathcal A_1, \ldots, \mathcal A_M$, $\mathcal B_1, \ldots, \mathcal B_N$ and $k_f,k_b>0$ are the forward and backward reaction rate constants. For simplicity, yet without loss of generality, we assume $k_f=k_b=1$. By applying mass action kinetics to \eqref{revrec} and by using the short notation
\begin{equation*}
		a = (a_1, \ldots, a_M), \quad b = (b_1, \ldots, b_N), \quad \alpha = (\alpha_1, \ldots, \alpha_M), \quad \beta = (\beta_1, \ldots, \beta_N),
\end{equation*}
\begin{equation*}
		a^\alpha = \prod_{i=1}^Ma_i^{\alpha_i}, \quad\quad b^{\beta} = \prod_{j=1}^{N}b_j^{\beta_j},
\end{equation*}
we study the following reaction-diffusion system:
	\begin{equation}\label{R}\tag{R}
	\begin{cases}
	\begin{aligned}
	\partial_t a_i - d_i\Delta(a_i^{m_i}) 
	&= f_i(a,b):= -\alpha_i\left[a^\alpha - b^\beta \right],\ \forall i=1,\ldots,M  &&\quad x\in\Omega, \quad\ t>0,\\
	\partial_t b_j - h_j\Delta(b_j^{p_j}) 
	&= g_j(a,b):= \beta_j\left[a^\alpha - b^\beta \right],\ \forall j=1,\ldots,N  &&\quad x\in\Omega, \quad\ \, t>0,\\
	d_i\nabla (a_i^{m_i})\cdot \overrightarrow{n} &= 0, \quad \forall i=1,\ldots,M,&&\quad x\in\partial\Omega, \quad t>0,\\
	h_j\nabla (b_j^{p_j})\cdot \overrightarrow{n} &= 0, \quad \forall j=1,\ldots,N,\quad &&\quad x\in\partial\Omega, \quad t>0,\\
	a_i(x,0) &= a_{i,0}(x),\quad \forall i=1,\ldots,M,&&\quad x\in\Omega,\\
	b_j(x,0) &= b_{j,0}(x),\quad\forall j=1,\ldots,N, &&\quad x\in\Omega.
	\end{aligned}
	\end{cases}
	\end{equation}
Here $d_i, h_j >0$ are diffusion coefficients and $m_i, p_j>1$ are nonlinear diffusion exponents. It is clear that \eqref{R} is a special case of \eqref{S}. It is also straightforward to verify condition \eqref{P}, while condition \eqref{G} is satisfied by choosing,
	$$
	\nu=\max\Biggl\{\sum_{i=1}^M\alpha_i,\sum_{j=1}^N\beta_j\Biggr\}.
	$$
Finally condition \eqref{M} is a consequence from noting that 
\[
	\frac{1}{M}\sum_{i=1}^M\frac{1}{\alpha_i	}f_i(a,b)
	+
	\frac{1}{N}\sum_{j=1}^N\frac{1}{\beta_j	}g_j(a,b)=0. 
\]

After having the conditions \eqref{P}, \eqref{G} and \eqref{M} verified, 
Theorem \ref{thm:1.1} implies the existence of global weak nonnegative solutions of system \eqref{R} provided
	\begin{equation*}
	m_i, p_j > \max\left\{\nu - 1; 1\right\} \quad \text{ for all } \quad i=1\ldots M, \; j=1\ldots N.
	\end{equation*}
Moreover by Theorem \ref{T1.2}, these solutions are bounded in dimensions $d=1,2$, or in dimensions $d\geq 3$ when  additionally assuming
\begin{equation*}
	m_i, p_j > \nu - \frac{4}{d+2} \quad \text{ for all } \quad i=1\ldots M, \; j=1\ldots N.
\end{equation*}
By multiplying the equations for $a_i$ and $b_j$ with $\beta_j$ and $\alpha_i$, respectively, and by adding the resulting terms, integration by parts with the homogeneous Neumann boundary conditions implies that these solutions satisfies the following mass conservation laws:
	\begin{equation}\label{conservation-laws}
	\beta_j\int_\Omega a_i(x,t)dx+\alpha_i\int_\Omega b_j(x,t)dx
	=
	\beta_j\int_\Omega a_{i,0}(x)dx+\alpha_i\int_\Omega b_{j,0}(x)dx
	=:M_{ij}>0,\qquad \forall i,j,
	\end{equation}
amongst which exactly $M+N-1$ linearly independent conservation laws ought to be selected 
and only the corresponding $M+N-1$ components of the initial mass vector $M_{ij}$ 
need to be calculated from the initial data. 

System \eqref{R} possesses for each fixed positive initial mass vector $(M_{ij})$ a unique positive detailed balanced equilibrium $(a_\infty,b_\infty) = (a_{1,\infty}, \ldots, a_{M,\infty}, b_{1,\infty}, \ldots, b_{N,\infty}) \in (0,\infty)^{M+N}$, which is the solutions of the following equilibrium equations:
	\[
	\begin{cases}
	\prod^M_{i=1}a_{i\infty}^{\alpha_i}=\prod^N_{j=1}b_{j\infty}^{\beta_j},
	\\
	\beta_ja_{i\infty}+\alpha_ib_{j\infty}=M_{ij},\quad \forall i,j,
	\end{cases}
	\]
where we recall that the second line constitutes of only  $M+N-1$ linearly independent conditions. 
\medskip

To study the convergence to equilibrium for \eqref{R}, we will use the so-called entropy method, which recently proved a highly suitable tool in the analysis of the large-time-behaviour of dissipative PDE systems. With respect to reaction-diffusion systems with linear diffusion, we refer in particular to \cite{DF06,DF07,DF08,MHM15,DFT16,FT17a,FT17}. 

The key \textit{entropy functional} (or in this case the free energy functional) of system \eqref{R} is defined by
\begin{equation*}
E[a,b] = \sum_{i=1}^{M}\int_{\Omega}(a_i\ln a_i - a_i + 1)dx + \sum_{j=1}^{N}\int_{\Omega}(b_j\ln b_j - b_j + 1)dx
\end{equation*}
which dissipates according to   the nonnegative {\it entropy production functional}, that is formally
\begin{equation*}
-\frac{d}{dt}E[a,b]=: D[a,b] 
= 
\sum_{i=1}^{M}d_i\int_{\Omega}\frac{|\nabla a_i|^2}{a_i^{2-m_i}}dx 
+	\sum_{j=1}^{N}h_j\int_{\Omega}\frac{|\nabla b_j|^2}{b_j^{2-p_j}}dx
+ \int_{\Omega}(a^\alpha- b^\beta)\ln{\frac{a^\alpha}{b^\beta}}dx
\ge 0.
\end{equation*}
In the case of linear diffusion, i.e. $m_i = p_j = 1$ for all $i=1\ldots M, j=1\ldots N$, the convergence to equilibrium of solutions of \eqref{R} (or some special cases) was recently studied in e.g. \cite{DF06,DF08,MHM15,FT17a,PSZ16}. 

Let us briefly review the entropy method used in the case of linear diffusion and then highlight the difficulties to be overcome in the current paper when dealing with nonlinear diffusion. In the case of  linear diffusion, the entropy production writes as
\begin{equation*}
D_{lin}[a,b] 
= 
\sum_{i=1}^{M}d_i\int_{\Omega}\frac{|\nabla a_i|^2}{a_i}dx 
+	\sum_{j=1}^{N}h_j\int_{\Omega}\frac{|\nabla b_j|^2}{b_j}dx
+ \int_{\Omega}(a^\alpha- b^\beta)\ln{\frac{a^\alpha}{b^\beta}}dx
\ge 0
\end{equation*}
and the entropy method consists in establishing a functional inequality of the form
\begin{equation}\label{eede}
	D_{lin}[a,b] \geq \lambda (E[a,b] - E[\ainf,\binf])
\end{equation}
for all functions $a = (a_i)$, $b= (b_j)$ satisfying the conservation laws \eqref{conservation-laws}. In order to do that, one first uses an additivity property of the relative entropy to calculate
\begin{equation*}
	\begin{aligned}
		E[a,b] - E[\ainf,\binf] & = \left[\sum_{i=1}^{M}\int_{\Omega}a_i\log{\frac{a_i}{\overline{a}_i}}dx + \sum_{j=1}^{N}\int_{\Omega}b_j\log{\frac{b_j}{\overline{b}_j}}dx\right]
	\\
		& \qquad + \left[\sum_{i=1}^{M}(\overline{a}_i\log{\frac{\overline{a}_i}{a_{i,\infty}}} - \overline{a}_i + a_{i,\infty}) + \sum_{j=1}^{N}(\overline{b}_j\log{\frac{\overline{b}_j}{b_{j,\infty}}} - \overline{b}_j + b_{j,\infty}) \right]\\
		&=: I_1 + I_2.
	\end{aligned}
\end{equation*}
The term $I_1$ is controlled in terms of the entropy production $D_{lin}[a,b]$ thanks to the Logarithmic Sobolev Inequality (LSI)
\begin{equation}\label{LSI}
	\int_{\Omega}\frac{|\nabla f|^2}{f}dx \geq C_{\mathrm{LSI}}\int_{\Omega}f\log\frac{f}{\overline{f}}dx \quad \text{ for all } \quad 0 \leq f \in H^1(\Omega).
\end{equation}
The remain term $I_2$ only involves the averages of the concentrations $\overline{a}_i, \overline{b}_j$ and can be controlled by $D_{lin}[a,b]$ through lengthly, technical, but constructive estimates (see e.g. \cite{FT17a,PSZ16} for more details). Note that this entropy approach applies successfully to more complex chemical reaction networks than \eqref{R}, see \cite{MHM15,DFT16,FT17,Mie}. We emphasised that the Logarithmic Sobolev Inequality \eqref{LSI} is not only used to control the term $I_1$ but also plays an important role in the estimates controlling the term $I_2$.
\medskip

In the case of nonlinear diffusion as here considered, we need a generalisation of the LSI \eqref{LSI} to exponents $m_i, p_j \ge1$. %One alternative is to look for generalisation of \eqref{LSI} which could be helpful. 
In this paper, we utilise the following generalisation (see e.g. \cite{MM17}): for any $m > (d-2)_+/d$ with $(d-2)_+ = \max\{d-2;0\}$, there exists a constant $C(\Omega,m)>0$ such that
\begin{equation*}
	\int_{\Omega}\frac{|\nabla f|^2}{f^{2-m}}dx \geq C(\Omega,m)\,\overline{f}^{\,m-1}\int_{\Omega}f\log\frac{f}{\overline{f}}dx.
\end{equation*}
When $m=1$, this coincides with the classical Logarithmic Sobolev inequality \eqref{LSI}. For system \eqref{R}, we have in  particular 
\begin{equation}\label{general}
	\int_{\Omega}\frac{|\nabla a_i|^2}{a_i^{2 - m_i}}dx \geq C(\Omega, m_i)\,\overline{a}_i^{\,m_i - 1}\int_{\Omega}a_i\log\frac{a_i}{\overline{a}_i}dx  \quad \text{ and } \quad \int_{\Omega}\frac{|\nabla b_j|^2}{b_j^{2-p_j}}dx \geq C(\Omega,p_j)\,\overline{b}_j^{\,p_j - 1}\int_{\Omega}b_j\log\frac{b_j}{\overline{b}_j}dx.
\end{equation}
Note that if we assume the averages $\overline{a}_i$ and $\overline{b}_j$ to be bounded below by a positive constant, then one can apply the same strategy as for the linear diffusion case in order to obtain the convergence to equilibrium. However, there is no chemical/physical reason for such a lower bound to hold in the transient behaviour of system \eqref{R} subject to general initial data. There are even perfectly admissible initial conditions, where some averages are zero since 
the corresponding species have not yet been formed.

To overcome this difficulty, we first observe that the mass conservation laws \eqref{conservation-laws} subject to a positive mass vector $M_{i,j}>0$ implies that the averages $\overline{a}_i$ and $\overline{b}_j$ cannot be simultaneously small. Thus, at any  fixed time at least one of the inequalities in \eqref{general} is useful, since either $\overline{a}_i \geq \varepsilon$ or $\overline{b}_j \geq \varepsilon$ for some suitably chosen $\varepsilon>0$ depending on $M_{i,j}>0$. Secondly, we are able to compensate the still lacking lower bounds in %the other inequalities of 
\eqref{general}
by a phenomena which can be called "indirect diffusion effect"  and which means in our context that the reversible reaction \eqref{revrec} 
transfers diffusion from a species $a_i$ (with strictly positive diffusion bound in \eqref{general} due to $\overline{a}_i \geq \varepsilon$) 
to other species $b_j$ (with lacking positive lower diffusion bound) in terms of a functional inequality, see Lemma \ref{IndDiff} below. 
%transfer, in some sense, the diffusion effect from one or more concentrations into other concentrations involving the reactions.
%the terms with small averages using the combination of the terms with larger averages and the reversible reaction. Such 
%a phenomena is usually called "indirect diffusion effect" which means that the reversible reactions transfer, in some sense, the diffusion effect from one or more concentrations into other concentrations involving the reactions. 

Examples of indirect diffusion effect inequalities were already derived in e.g. \cite{DF07,FLT17,FPT17}, yet typically with a proof which requires uniform in time $L^{\infty}$-bounds on the solutions, which is a severe technical restriction as 
$L^{\infty}$-bounds for general reaction-diffusion systems are often unknown due to the lack of comparison principles. {Note that also the $L^\infty$-bounds of Theorem \ref{thm:1.1} would be insufficient since polynomially growing  and not uniform in time.}
\medskip

In this work, we are able to prove an indirect diffusion functional inequality without using any $L^{\infty}$-bounds on solutions but instead by exploiting the special structure of \eqref{R}, see Lemma \ref{IndDiff}. Nevertheless, in the remaining part of 
applying the entropy method, 
the polynomial growth in time of the $L^{\infty}$-norm of Theorem \ref{T1.2} is still needed in one estimate concerning the relative entropy, yet the $L^{\infty}$-norm appears only within a 
logarithm. 
While it is unclear to us whether %the resulting logarithm of the polynomially growing $L^{\infty}$-bounds 
this is essential or just technical necessary in our approach,
it allows to derive a \textit{time-dependent} entropy-entropy production inequality (as a generalisation of the functional inequality \eqref{eede}) of the form
\begin{equation}\label{EEP}
	D[a(T),b(T)] \geq \Theta(T)(E[a(T),b(T)] - E[\ainf,\binf]) \quad \text{ for all }\quad T>0,
\end{equation}
where the function $\Theta: \mathbb R_+ \to \mathbb R_+$ is of order $1/\ln(1+T)$ and satisfies $\int_0^{+\infty}\Theta(\tau)d\tau = +\infty$. Thus, a classical Gronwall argument implies explicit {\it algebraic} decay of $E[a(T),b(T)] - E[\ainf,\binf]$ to zero and thus, algebraic convergence to equilibrium in relative entropy. 

To obtain exponential from algebraic decay, we show that after some sufficiently large time $T_0>0$, the averages $\overline{a}_i(T)$ and $\overline{b}_j(T)$ are bounded below by a positive constant for all $T\geq T_0$ (since the equilibrium $(a_\infty,b_\infty)$ consists of positive constants). Hence, for $T\geq T_0$, we can use the inequalities \eqref{general} like in the case for systems with linear diffusion and obtain accordingly exponential convergence to equilibrium. Finally, since $T_0$ can be explicitly estimated, one recovers global exponential convergence to equilibrium (i.e. for all $T\geq 0$) at the price of a smaller, yet explicit constant.
Hence, the second main result of this paper is the following theorem.

	\begin{theorem}\label{T1.3}
	Let $\Omega\subset\mathbb{R}^d$ be bounded with sufficiently smooth boundary. Consider system \eqref{R} -- which satisfies the conditions \eqref{G},\eqref{M} and \eqref{P} -- subject to non-negative initial data $a_{i,0},b_{j,0}\in L^\infty(\Omega)$. Assume for all $i=1\ldots M, j=1\ldots N$ that
	\[
	m_i, p_j > \max\{\nu-1;1\}, \qquad \text{where } \quad 
	\nu=\max\Biggl\{\sum_{i=1}^M\alpha_i,\sum_{j=1}^N\beta_j\Biggr\}.
	\]
	Moreover, in dimensions $d\geq 3$, we additionally assume 
		\[
		m_i,p_j>\nu - \frac{4}{d+2}, \qquad \text{ for all } \quad i=1\ldots M,\ j=1\ldots N.
		\]
Finally, consider a positive initial mass vector $M_{ij}>0$, which uniquely determines
a positive equilibrium $(a_{i\infty},b_{j\infty})$ of system \eqref{R}.

Then, the bounded global weak solutions of Theorem \ref{T1.2} converge exponentially to $(a_\infty,b_\infty)$ in all $L^p$-norms for $1\leq p<\infty,$ that is
		\[
		\sum^M_{i=1}\|a_i(t)-a_{i\infty}\|_{L^p(\Omega)}
		+
		\sum^N_{j=1}\|b_j(t)-b_{j\infty}\|_{L^p(\Omega)}
		\leq
		C\, e^{-\lambda_pt}
		\]
		where the constant $C>0$ and the convergence rate $\lambda_p>0$  can be computed explicitly.
	\end{theorem}
	
	\noindent{\bf Notation:}
	\begin{itemize}
		\item We denote by $\|\cdot\|$ the usual norm of $L^2(\Omega)$. For other $1\leq p < +\infty$, we write $\|\cdot\|_p$ as the norm of $L^p(\Omega)$.
		\item For any $T>0$, $Q_T = \Omega\times (0,T)$ and $L^p(Q_T) = :L^p(0,T;L^p(\Omega))$. The space-time norm is defined as usual
		\begin{equation*}
		\|f\|_{L^p(Q_T)}^p = \int_{0}^T\!\!\int_{\Omega}|f(x,t)|^pdxdt.
		\end{equation*}
		\item Throughout this work, we will denote by $C_T$ a generic positive constant which depends on certain parameters, and more importantly $C_T$ grows at most polynomially, i.e. there exists a polynomial $P(x)$ such that $C_T \leq P(T)$ for all $T>0$.
	\end{itemize}
	
	\medskip
	\noindent{\bf Organisation of the paper:} {Section \ref{boundedness} states the proof of Theorem \ref{T1.2}. The proof of Theorem \ref{T1.3} is detailed in Section \ref{sec:conv}. This proof uses also a previously proven entropy-entropy production estimate for reaction-diffusion systems with linear diffusion, which is recalled in Section \ref{sec:add} for the sake of completeness. Finally, the existence of global weak solution is stated in Section \ref{app}.
%In Appendix A, we prove the entropy-entropy production estimate for the case of linear diffusion and the global existence of weak solutions is proved in Appendix B.
}
	
	%%%%%%%%%%%%%%%%%%%%%%%%%%%%%%%%%%%%%%%%%%%%%%%%%%%%%%%%%%%%%%%%%%%%%%%%%%%%%%%%%%%%%%%%%%%%%%%%%%%%%%%%%%%%%%%%%%%%%%%%%%%%%%%%%%%%%%%%%%%%%%%%%%%%%%%%%%%%%%%%%%%%%%%%%%%%%%%%%%%%%%%%%%%%%%%%%%%%%%

\section{Boundedness and local continuity of weak solutions}\label{boundedness}
	In this section, we prove for sufficiently large diffusion exponents $m_i$ that the weak solutions obtained in Theorem \ref{thm:1.1} are actually bounded in $L^{\infty}$ and thus locally H\"older continuous. In Lemma \ref{heat-regularity}, we device a bootstrap argument for the inhomogeneous porous media equation which proves that if the porous media exponents $m_i$ and the initial integrability are high enough, then the weak solutions of Theorem \ref{thm:1.1} %satisfying the a-priori estimate \eqref{duality} 
satisfy an improve integrability in a space $L^{s}(Q_T)$ and the $L^{s}$-norm grows at most polynomially in time $T$.
	
	\begin{lemma}[Smoothing effect of porous medium equation]\label{heat-regularity}\hfill\\
		Suppose that $m\geq 1$. Assume $f\in L^{p_0}(Q_T)$ for some $p_0>1$ with $\|f\|_{L^{p_0}(Q_T)} \leq C_T$. Let $u$ be a solution to the inhomogeneous porous medium equation with positive diffusion coefficient $\delta >0$ 
		\begin{equation}\label{porous}
		\begin{cases}
		\partial_t u - \delta\Delta(|u|^{m-1}u) = f, &x\in\Omega, \qquad t>0,\\
		\delta\nabla (|u|^{m-1}u)\cdot \overrightarrow{n} = 0, &x\in\partial\Omega, \quad\ t>0,\\
		u(x,0) = u_0(x), &x\in\Omega,
		\end{cases}
		\end{equation}
		%with a positive diffusion coefficient $\delta >0$,
		and subject to initial data $u_0 \in L^{\infty}(\Omega)$. 
%		Moreover, we assume that $u$ satisfies the uniform $L^1$-bound
%		\begin{equation*}
%		\|u(t)\|_{1} \leq M \quad \text{ for all } \quad t\geq 0.
%		\end{equation*}
		Then, $u$ satisfies 
		\[
		\|u\|_{L^r(Q_T)}\leq C_T,\quad\forall r\in[1,s),
		\]
		where
		\[
		s=
		\begin{cases}
		+\infty,\quad 
		&\text{if}\quad p_0\geq\frac{d+2}{2},
		\\
		\frac{(md+2)p_0}{d+2-2p_0},\quad 
		&\text{if}\quad p_0<\frac{d+2}{2},	
		\end{cases}
		\]
and with a constant $C_T$, which only depends on $q, d, m, \Omega$ and at most polynomially on $T$.
	\end{lemma}
\begin{remark}
In the linear case $m=1$ Lemma \ref{heat-regularity} recovers the corresponding regularity estimates of the heat equation, see \cite{CDF14}. 
While the smoothing effect stated in Lemma \ref{heat-regularity} is 
certainly well-known, our main contribution here lies in the
polynomial growth in time of the norms, which will be crucial in Section \ref{sec:conv}.
%We emphasise that the smoothing effect in this lemma is well known (see e.g. \cite{Vaz07})\textcolor{red}{Vazquez's book, Section 17.7, provides only $L^{\infty}(0,T;L^{p}(\Omega))$ estimates instead of $L^p(Q_T)$ estimates. Need a more proper reference here?}. The novelty here
%, as shown also for the case of linear diffusion in \cite{CDF14}, 
%is the polynomial growth of the norm w.r.t. time $T$, which we crucially require  in the Section \ref{sec:conv}.
\end{remark}
	\begin{proof}
		The idea of the proof of this lemma follows \cite[Lemma 3.3]{CDF14} and is divided into several steps.
		
		\medskip
\noindent{\bf Step 1.} Let $\mu> 1$. By multiplying \eqref{porous} by $\mu |u|^{\mu-1}\sign(u)$ then integrating over $\Omega$, we obtain
\begin{equation}\label{e1}
\frac{d}{dt}\|u\|_{\mu}^{\mu} -\delta\mu\int_{\Omega}\Delta(|u|^{m-1}u) |u|^{\mu-1}\sign(u)dx = 
\mu\int_{\Omega}f |u|^{\mu-1}\sign(u)dx.
\end{equation}
Integration by parts and the homogeneous Neumann boundary condition $\nabla(|u|^{m-1}u)\cdot \overrightarrow{n} = 0$ lead to
\begin{equation*}
\begin{aligned}
-\delta\mu\int_{\Omega}\Delta(|u|^{m-1}) |u|^{\mu-1}\sign(u)dx &= {m(\mu-1)\mu}\delta\int_{\Omega}|u|^{m+\mu-3}|\nabla u|^2dx
+ m \mu \delta\int_{\Omega} |u|^{m+p-2}|\nabla u|^2 \delta(u)dx\\
		&\ge  \underbrace{\frac{4m(\mu-1)\mu\delta}{(m+\mu-1)^2}}_{=:{C(\mu)}}\int_{\Omega}\left|\nabla\left(|u|^{\frac{m+\mu-1}{2}}\right)\right|^2dx.
		\end{aligned}
		\end{equation*}
		By Young's inequality
		\begin{equation*}
		\left|\mu\int_{\Omega}f |u|^{\mu-1}\sign(u)dx\right| \leq \mu\|f\|_{p_0}\|u\|_{\frac{p_0(\mu-1)}{p_0-1}}^{\mu-1}.
		\end{equation*}
		Therefore, it follows from \eqref{e1} that
		\begin{equation}\label{eq-mu}
		\frac{d}{dt}\|u\|_{\mu}^{\mu} + C(\mu)\int_{\Omega}\left|\nabla\left(|u|^{\frac{m+\mu-1}{2}}\right)\right|^2dx \leq \mu\|f\|_{p_0}\|u\|_{\frac{p_0(\mu-1)}{p_0-1}}^{\mu-1}.
		\end{equation}
		
		\medskip
\noindent{\bf Step 2.} Choose $\mu = p_0>1$ in \eqref{eq-mu}, we get
		\begin{equation}\label{eq-p0}
		\frac{d}{dt}\|u\|_{p_0}^{p_0} + C(p_0)\int_{\Omega}\left|\nabla\left(|u|^{\frac{m+p_0-1}{2}}\right)\right|^2dx \leq p_0\|f\|_{p_0}\|u\|_{p_0}^{p_0-1}.
		\end{equation}
		By applying for $r<1$ the elementary inequality 
		\begin{equation}\label{elementary}
		y' \leq \alpha(t)y^{1-r} \quad \Longrightarrow \quad y(T) \leq \left[y(0)^{r}+r\int_0^T\alpha(t)dt\right]^{1/r},
		\end{equation}
		to \eqref{eq-p0} with $r=1/p_0$ and $y(t) = \|u(t)\|_{p_0}^{p_0}$, we obtain
		\begin{equation}\label{CT0}
		\|u(T)\|_{p_0}^{p_0} \leq \left[\|u_0\|_{p_0} + \int_0^{T}\|f\|_{p_0}dt \right]^{p_0} \leq \left[\|u_0\|_{p_0} + \|f\|_{L^{p_0}(Q_T)}T^{(p_0-1)/p_0} \right]^{p_0} =: C_{T,0}.
		\end{equation}
		That means
		\begin{equation}\label{p0}
		u \in L^{\infty}(0,T;L^{p_0}(\Omega)) \quad \text{ and } \quad \|u(T)\|_{p_0}^{p_0} \leq C_{T,0}
		\end{equation}
		with $C_{T,0}$ is defined in \eqref{CT0} grows at most polynomially in $T$. By integrating \eqref{eq-p0} with respect to $t$ on $(0,T)$ and by using Young's inequality and the convention $r_0 := m+p_0 - 1>1$, we get
		\begin{equation*}
		\begin{aligned}
		C(p_0)\iTO\left|\nabla\left(|u|^{\frac{r_0}{2}}\right)\right|^2dxdt &\leq \|u_0\|_{p_0}^{p_0} + p_0\int_0^T\|f\|_{p_0}\|u\|_{p_0}^{p_0-1}dt\\
		&\leq \|u_0\|_{p_0}^{p_0} + p_0\|f\|_{L^{p_0}(Q_T)}\|u\|_{L^{p_0}(Q_T)}^{p_0-1}.
		\end{aligned}
		\end{equation*}
		By adding $C(p_0)\iTO\left||u|^{\frac{r_0}{2}}\right|^2dxdt$ to both sides, we have
		\begin{equation}\label{h0}
		\begin{aligned}
		C(p_0)\int_0^T\left\||u|^{\frac{r_0}{2}}\right\|_{H^1(\Omega)}^2dt &= C(p_0)\int_0^T\left[
		\int_{\Omega}\left|\nabla\left(|u|^{\frac{r_0}{2}}\right)\right|^2dx + \int_{\Omega}\left| |u|^{\frac{r_0}{2}} \right|^2 dx\right]dt\\
		&\leq \|u_0\|_{p_0}^{p_0} + p_0\|f\|_{L^{p_0}(Q_T)}\|u\|_{L^{p_0}(Q_T)}^{p_0-1} + C(p_0)\int_0^T\|u\|_{r_0}^{r_0}dt.
		\end{aligned}
		\end{equation}	
		%\textcolor{red}{Question: Can be obtain slower growth in $T$ when we add a lower $L^p$ norm for the Sobolv embedding? I can't see how.}
		
		By the Sobolev's embedding, we have
		\begin{equation}\label{h1}
		C(p_0)\int_0^T\left\||u|^{\frac{r_0}{2}}\right\|^2_{H^1(\Omega)} \geq C(p_0)\,C_S^2\int_0^T\|u\|_{s_0}^{r_0}dt \quad \text{ with } \quad s_0 = \begin{cases}\frac{r_0d}{d-2} &\text{ if } d \geq 3,\\r_0 < s_0 < \infty \text{ arbitrary } &\text{ if } d = 1,2.\end{cases}
		\end{equation}
		On the other hand, by using the bound $\|u(t)\|_{p_0}^{p_0} \leq C_{T,0}$ in \eqref{p0} and the interpolation inequality 
		\begin{equation*}
		\|u\|_{r_0} \leq \|u\|_{p_0}^{\gamma}\|u\|_{s_0}^{1-\gamma} \leq C_{T,0}^{\gamma/p_0}\|u\|_{s_0}^{1-\gamma} \quad \text{with} \quad \frac{1}{r_0} = \frac{\gamma}{p_0} + \frac{1-\gamma}{s_0}
		\quad \text{for} \quad \gamma = \frac{2p_0}{2p_0+(m-1)d}\in(0,1],
		\end{equation*}
		we estimate {in the cases $m>1$ for which $\gamma<1$}
		\begin{equation}\label{h2}
		C(p_0)\int_0^T\|u\|_{r_0}^{r_0}dt \leq C(p_0)\int_{0}^{T}C_{T,0}^{\gamma r_0/p_0} \|u\|_{s_0}^{(1-\gamma)r_0}dt \leq \frac{C(p_0)\,C_S^2}{2}\int_0^T\|u\|_{s_0}^{r_0}dt + CC_{T,0}^{{r_0}/p_0}T,
		\end{equation}
		where we have used Young's inequality (with the exponents $1=(1-\gamma) + \gamma$) in the last step. {Note that if $m=1$, the bound \eqref{h2} holds still true yet without the first term and with $r_0/p_0=1$.} 
		Inserting \eqref{h1} and \eqref{h2} into \eqref{h0} leads to
		\begin{equation}\label{DT0}
		\begin{aligned}
		\int_0^T\|u\|_{s_0}^{r_0}dt &\leq \frac{2}{C(p_0)\,C_S^2}\left[\|u_0\|_{p_0}^{p_0} + p_0\|f\|_{L^{p_0}(Q_T)}\|u\|_{L^{p_0}(Q_T)}^{p_0-1}  + CC_{T,0}^{{r_0}/p_0}T\right]\\
		&\leq \frac{2}{C(p_0)\,C_S^2}\left[\|u_0\|_{p_0}^{p_0} + p_0\|f\|_{L^{p_0}(Q_T)}\left(TC_{T,0}\right)^{\frac{p_0-1}{p_0}}  + CC_{T,0}^{{r_0}/p_0}T\right] =: D_{T,0} \quad (\text{use (\ref{p0}))}.
		\end{aligned}
		\end{equation}
%		\begin{remark}
%			Ideas to control last term on the r.h.s.: 
%			i) $L^{\infty}_t L^1_x$ interpolation as in [FLS], which finals gives only a constant on the r.h.s. 
%			ii) using $L^{\infty}_t L^{p_0}_x$ derived in \eqref{p0}, but this will give a term $C_T$ on the r.h.s. 
%			We chose i) because we later need at most linear grow in time of the constant.  \textcolor{red}{Do we still need this remark?}
%		\end{remark}
		
		%$H^1(\Omega)\hookrightarrow L^{2N/(N-2)}(\Omega)$ 
		It follows that
		\begin{equation}\label{s0}
		u \in L^{r_0}(0,T;L^{s_0}(\Omega)) \quad \text{ with } \quad \begin{cases}s_0 = \frac{r_0d}{d-2} &\text{ if } d\geq 3,\\
		r_0< s_0 <\infty \text{ arbitrary } &\text{ if } d=1,2,
		\end{cases}
		\end{equation}
		and
		\begin{equation*}
		\int_0^T\|u\|_{s_0}^{r_0}dt \leq D_{T,0}
		\end{equation*}
		with $D_{T,0}$ defined in \eqref{DT0}. 	
		
		Next, we construct a sequence $p_n \geq 1$
		based on the estimate \eqref{p0} and \eqref{s0} such that
		\begin{equation}\label{pn}
		\|u(T)\|_{p_n}^{p_n} \leq C_{T,n}
		\end{equation}
		and
		\begin{equation}\label{sn}
		\int_0^T\|u\|_{s_n}^{r_n}dt \leq D_{T,n} \quad \text{ with }\quad r_n = m + p_n - 1 \quad \text{ and } \quad  
		\begin{cases}s_n = \frac{r_nd}{d-2} &\text{ if } d\geq 3,\\
		r_n < s_n <\infty \text{ arbitrary } & \text{ if } d=1,2,
		\end{cases}
		\end{equation}
		in which $C_{T,n}$ and $D_{T,n}$ are constants growing at most polynomially in $T$.
		
		\medskip 
\noindent{\bf Step 3 (Iteration of \eqref{pn}).} In \eqref{eq-mu}, we set $\mu = p_{n+1}$ for $p_{n+1}$ to be chosen later. Thus, we have
		\begin{equation}\label{eq-pn+1}
		\frac{d}{dt}\|u\|_{{p_{n+1}}}^{{p_{n+1}}} + C({p_{n+1}})\int_{\Omega}\left|\nabla\left(|u|^{\frac{r_{n+1}}{2}}\right)\right|^2dx \leq p_{n+1}\|f\|_{p_0}\|u\|_{\frac{p_0({p_{n+1}}-1)}{p_0-1}}^{{p_{n+1}}-1},
		\end{equation}	
		where we recall that $r_{n+1} = m + p_{n+1}-1$. By $L^p$- interpolation, we have
		\begin{equation*}
		\|u\|_{\frac{p_0({p_{n+1}}-1)}{p_0-1}} \leq \|u\|_{p_{n+1}}^{1-\theta}\|u\|_{s_n}^{\theta}
		\end{equation*}
		and where $p_{n+1}>1$ has to be chosen such that $\frac{p_0(p_{n+1}-1)}{p_0-1}\in(p_{n+1},s_n)$ with $p_{n+1}<s_n$, which entails $\theta\in (0,1)$ in
		\begin{equation}\label{theta-1}
		\frac{p_0-1}{p_0(p_{n+1}-1)} = \frac{1-\theta}{p_{n+1}} + \frac{\theta}{s_n}.
		\end{equation}
		Note that $\frac{p_0(p_{n+1}-1)}{p_0-1} > p_{n+1}$ is always satisfied provided that $p_{n+1}>p_0$, i.e. that the sequence $p_n$ is strictly monotone increasing.
		
		It then follows from \eqref{eq-pn+1} (by neglecting the second term on the left hand side) that
		\begin{equation*}
		\frac{d}{dt}\|u\|_{p_{n+1}}^{p_{n+1}} \leq p_{n+1}\|f\|_{p_0}\|u\|_{s_n}^{\theta(p_{n+1}-1)}\left(\|u\|_{p_{n+1}}^{p_{n+1}}\right)^{1 - \frac{1+\theta(p_{n+1}-1)}{p_{n+1}}}.
		\end{equation*}
		By applying again the elementary inequality \eqref{elementary} with $y(t) = \|u(t)\|_{p_{n+1}}^{p_{n+1}}$ and $r = \frac{1+\theta(p_{n+1}-1)}{p_{n+1}} < 1$, it yields
		\begin{equation}\label{pnn}
		\begin{aligned}
		\|u(T)\|_{p_{n+1}}^{p_{n+1}}
		&\leq \left[\|u_0\|_{p_{n+1}}^{1+\theta(p_{n+1}-1)} + (1+\theta(p_{n+1}-1))\int_0^T\|f\|_{p_0}\|u\|_{s_n}^{\theta(p_{n+1}-1)}dt \right]^{\frac{p_{n+1}}{1+\theta(p_{n+1}-1)}}\\
		&\leq \left[\|u_0\|_{p_{n+1}}^{1+\theta(p_{n+1}-1)} + (1+\theta(p_{n+1}-1))\|f\|_{L^{p_0}(Q_T)}\biggl(\int_0^T\|u\|_{s_n}^{\theta(p_{n+1}-1)\frac{p_0}{p_0-1}}dt\biggr)^{\!\frac{p_0-1}{p_0}} \right]^{\frac{p_{n+1}}{1+\theta(p_{n+1}-1)}}.
		\end{aligned}
		\end{equation}
		In order to continue estimating by using \eqref{sn}, we choose $p_{n+1}$ as
		\begin{equation}\label{theta-2}
		\theta(p_{n+1}-1)\frac{p_0}{p_0-1} = r_n.
		\end{equation}
		Since $r_n= s_n\frac{d-2}{d}$, 
		eq. \eqref{theta-2} implies $ \frac{\theta}{s_n} = (1-\frac{2}{d})\frac{p_0-1}{p_0(p_{n+1}-1)}$ and thus with \eqref{theta-1}
		\begin{equation}\label{thetan}
		\theta = 1 - \frac{2}{d} \frac{p_0-1}{p_0}\frac{p_{n+1}}{p_{n+1}-1}<1.
		\end{equation}
		In order to verify that above choice of $p_{n+1}$ satisfies 
		$\frac{p_0(p_{n+1}-1)}{p_0-1}<s_n$, we insert \eqref{thetan} into \eqref{theta-2} and calculate 
		\begin{equation*}
		(p_{n+1}-1)\frac{p_0}{p_0-1}-\frac{2}{d} p_{n+1} = s_n \frac{d-2}{d} \quad\Rightarrow\quad
		s_n-\frac{p_0(p_{n+1}-1)}{p_0-1} = \frac{2}{d}(s_n-p_{n+1})>0.
		\end{equation*}
		Similar, by recalling $s_n \frac{d-2}{d}=r_n=m-1+p_n$, we get 	the iteration
		\begin{equation}\label{iteration}
		p_{n+1} = p_n\frac{d(p_0-1)}{p_0(d-2) + 2} + \frac{d[(m-1)(p_0-1) + p_0]}{p_0(d-2)+2}.
		\end{equation}

		Altogether, by inserting \eqref{theta-2} into  \eqref{pnn}, we obtain thanks to \eqref{sn}
		\begin{equation}\label{CTn+1}
		\begin{aligned}
		\|u(T)\|_{p_{n+1}}^{p_{n+1}} &\leq \left[\|u_0\|_{p_{n+1}}^{^{1+\theta(p_{n+1}-1)}} + (1+\theta(p_{n+1}-1))\|f\|_{L^{p_0}(Q_T)}\left(\int_0^T\|u\|_{s_n}^{r_n}dt\right)^{\frac{p_0-1}{p_0}} \right]^{\frac{p_{n+1}}{1+\theta(p_{n+1}-1)}}\\
		&\leq \left[\|u_0\|_{p_{n+1}}^{^{1+\theta(p_{n+1}-1)}} + (1+\theta(p_{n+1}-1))\|f\|_{L^{p_0}(Q_T)}D_{T,n}^{\frac{p_0-1}{p_0}} \right]^{\frac{p_{n+1}}{1+\theta(p_{n+1}-1)}} =: C_{T,n+1}
		\end{aligned}
		\end{equation}	
		and thus
		\begin{equation}\label{bound-n+1}
		u \in L^{\infty}(0,T;L^{p_{n+1}}(\Omega)) \quad \text{ and } \quad \|u(T)\|_{p_{n+1}}^{p_{n+1}} \leq C_{T,n+1}.
		\end{equation}
		
		\medskip
\noindent{\bf Step 4 (Iteration of \eqref{sn}).} We will use similar arguments to {\bf Step 2}.	
		Integrating \eqref{eq-pn+1} and adding $\iTO\left||u|^{\frac{r_{n+1}}{2}}\right|^2dxdt$ to both sides yields in particular
		\begin{equation}\label{h3}
		\begin{aligned}
		 C(&p_{n+1})\int_0^T\left\||u|^{\frac{r_{n+1}}{2}}\right\|_{H^1(\Omega)}^2dt
		 = 		C({p_{n+1}})\int_0^T\int_{\Omega}\left[\left|\nabla\left(|u|^{\frac{r_{n+1}}{2}}\right)\right|^2dx + \left||u|^{\frac{r_{n+1}}{2}}\right|^2dx\right]dt &&\\
		&\leq \|u_0\|_{p_{n+1}}^{p_{n+1}} +  p_{n+1}\int_0^T\|f\|_{p_0}\|u\|_{\frac{p_0({p_{n+1}}-1)}{p_0-1}}^{{p_{n+1}}-1}dt + C({p_{n+1}})\int_0^T\|u\|_{r_{n+1}}^{r_{n+1}}dt&&\\
		&\leq \|u_0\|_{p_{n+1}}^{p_{n+1}} +p_{n+1}\int_0^T\|f\|_{p_0}\|u\|_{s_n}^{\theta(p_{n+1}-1)}\|u\|_{p_{n+1}}^{(1-\theta)(p_{n+1}-1)}dt + C({p_{n+1}})\int_0^T\|u\|_{r_{n+1}}^{r_{n+1}}dt &&(\theta \text{ in } \eqref{theta-1})\\
		&\leq \|u_0\|_{p_{n+1}}^{p_{n+1}} +{p_{n+1}}{C_{T,n+1}^{(1-\theta)\frac{(p_{n+1}-1)}{p_{n+1}}}}\int_0^T\|f\|_{p_0}\|u\|_{s_n}^{\theta(p_{n+1}-1)}dt+ C({p_{n+1}})\int_0^T\|u\|_{r_{n+1}}^{r_{n+1}}dt &&(\text{using } \eqref{bound-n+1})\\
		& \leq \|u_0\|_{p_{n+1}}^{p_{n+1}} +{p_{n+1}}{C_{T,n+1}^{(1-\theta)\frac{(p_{n+1}-1)}{p_{n+1}}}}\|f\|_{L^{p_0}(Q_T)}\biggl(\int_0^T\|u\|_{s_n}^{r_n}dt\biggr)^{\!\frac{p_0-1}{p_0}} + C({p_{n+1}})\int_0^T\|u\|_{r_{n+1}}^{r_{n+1}}dt&& (\text{using } \eqref{theta-2})\\
		& \leq \|u_0\|_{p_{n+1}}^{p_{n+1}} +{p_{n+1}}{C_{T,n+1}^{(1-\theta)\frac{(p_{n+1}-1)}{p_{n+1}}}}\|f\|_{L^{p_0}(Q_T)}D_{T,n}^{\frac{p_0-1}{p_0}} + C({p_{n+1}})\int_0^T\|u\|_{r_{n+1}}^{r_{n+1}}dt && (\text{using } \eqref{sn}).
		\end{aligned}
		\end{equation}
		Now by Sobolev's embedding
		\begin{multline}\label{h4}
		%\begin{gathered}
		C(p_{n+1})\int_0^T\left\||u|^{\frac{r_{n+1}}{2}}\right\|_{H^1(\Omega)}^2dt \geq C(p_{n+1})\,C_S^2\int_0^T\|u\|_{s_{n+1}}^{r_{n+1}}dt\\
		\text{ with } \quad s_{n+1} = \begin{cases}\frac{r_{n+1}d}{d-2} &\text{ if } d \geq 3,\\r_{n+1} < s_{n+1} < \infty \text{ arbitrary } &\text{ if } d = 1,2.\end{cases}
		%\end{gathered}
		\end{multline}
		By the bound $\|u(t)\|_{p_{n+1}}^{p_{n+1}} \leq C_{T,n+1}$, the interpolation inequality
		\begin{equation}\label{h5}
		\|u\|_{r_{n+1}} \leq \|u\|_{p_{n+1}}^{\gamma}\|u\|_{s_{n+1}}^{1-\gamma} \leq C_{T,n+1}^{\gamma/p_{n+1}}\|u\|_{s_{n+1}}^{1-\gamma}
		\end{equation}
		\begin{equation*}
			\text{ with } \quad \frac{1}{r_{n+1}} = \frac{\gamma}{p_{n+1}} + \frac{1-\gamma}{s_{n+1}} \quad\text{for}\quad \gamma = \frac{2p_{n+1}}{2p_{n+1} + (m-1)d}\in(0,1].
		\end{equation*}
		{Like in {\bf Step 2} in case $m>1$ and $\gamma<1$, we have	by Young's inequality,}
		\begin{equation*}
		\begin{aligned}
		C(p_{n+1})\int_0^T\|u\|_{r_{n+1}}^{r_{n+1}}dt &\leq C(p_{n+1})\int_{0}^{T}C_{T,n+1}^{\gamma r_{n+1}/p_{n+1}}\,\|u\|_{s_{n+1}}^{(1-\gamma)r_{n+1}}dt\\
		&\leq \frac{C(p_{n+1})\,C_S^2}{2}\int_{0}^{T}\|u\|_{s_{n+1}}^{r_{n+1}}dt + CTC_{T,n+1}^{{r_{n+1}}/p_{n+1}}
		\end{aligned}
		\end{equation*}
		analog to \eqref{h2} {while the case $m=1$ and $r_{n+1}/p_{n+1}=1$ follows without interpolation and the first term on the right-hand-side above}.
		Combining \eqref{h3}, \eqref{h4} and \eqref{h5} yields
		\begin{equation*}
		\frac{C(p_{n+1})\,C_S^2}{2}\int_0^T\|u\|_{s_{n+1}}^{r_{n+1}}dt \leq \|u_0\|_{p_{n+1}}^{p_{n+1}} +{p_{n+1}}{C_{T,n+1}^{(1-\theta)\frac{(p_{n+1}-1)}{p_{n+1}}}}\|f\|_{L^{p_0}(Q_T)}D_{T,n}^{\frac{p_0-1}{p_0}} + CTC_{T,n+1}^{{r_{n+1}}/p_{n+1}},
		\end{equation*}
		hence
		\begin{equation*}
		\int_0^T\|u\|_{s_{n+1}}^{r_{n+1}}dt \leq D_{T,n+1}
		\end{equation*}
		with 
		\begin{equation}\label{DTn+1}
		D_{T,n+1}:= \frac{2}{C(p_{n+1})C_S^2}\left[\|u_0\|_{p_{n+1}}^{p_{n+1}} +{p_{n+1}}{C_{T,n+1}^{(1-\theta)\frac{(p_{n+1}-1)}{p_{n+1}}}}\|f\|_{L^{p_0}(Q_T)}D_{T,n}^{\frac{p_0-1}{p_0}} + CTC_{T,n+1}^{{r_{n+1}}/p_{n+1}}\right].
		\end{equation}
		\medskip
\noindent{\bf Step 5.} Passing to the limit as $n\to\infty$. Considering the iteration \eqref{iteration}, the only possible fixed point $p_{\infty}$ of the  sequence $p_n$ is 
		\begin{equation*}
		p_{\infty} = \frac{d[(m-1)(p_0-1) + p_0]}{2[\frac{d+2}{2}-p_0]}.
		\end{equation*}
		Hence $p_{\infty}<0$ if and only if $p_0 > \frac{d+2}{2}$. 
		In particular, it is straightforward to check that the sequence 
		$p_n$ define by \eqref{iteration} is strictly monotone increasing if and only if either $p_n<p_{\infty}$ in the case $p_0<\frac{d+2}{2}$ or $p_n>p_{\infty}$ in the case $p_0>\frac{d+2}{2}$ when $p_{\infty}<0$ holds or $p_0=\frac{d+2}{2}$ where 
		$p_{\infty}=+\infty$.

		Therefore, we have as $n\to\infty$
		\begin{equation*}
		p_n \longrightarrow \begin{cases}
		p_{\infty} \quad \text{ if } \quad  p_0 < \frac{d+2}{2},\\
		+\infty \quad \text{ if } \quad p_0 \geq \frac{d+2}{2}.
		\end{cases}
		\end{equation*}
		
		\medskip
		\noindent{\bf Step 6 (Interpolation).} From \eqref{pn} and \eqref{sn} and by using the interpolation
		\begin{equation*}
		L^{\infty}(0,T;L^{p_n}(\Omega))\cap L^{r_n}(0,T;L^{s_n}(\Omega)) \hookrightarrow L^{\frac{N+2}{N}p_n + m-1}(Q_T)
		\end{equation*}
		we get $u\in L^{r}(Q_T)$ for all $r< \infty$ in the case $p_0 \geq \frac{d+2}{2}$. In the case $p_0 < \frac{d+2}{2}$, we obtain $u\in L^s(Q_T)$ for all 
		\begin{equation*}
		s < \frac{d+2}{d}p_{\infty} + m - 1 = \frac{(md+2)p_0}{d+2-2p_0}.
		\end{equation*}
		This completes the proof of Lemma \ref{heat-regularity}.
	\end{proof}
	\begin{lemma}\label{lem:S}  Let $u$ be a  weak solution to \eqref{S}  and 
		\[
		\|u\|_{L^{q_0}(Q_T)}\leq C_T,\quad
		\forall i=1,\ldots,S,\quad\text{with}\quad
		q_0>\frac{d(\nu-m)+2(\nu-1)}{2},
		\]
		where $m = \min\{m_i: i=1\ldots S\}$ and $\nu$ is defined in \eqref{G}.
		
		Then, it follows that $\|u_i\|_{L^\infty(Q_T)}\leq C_T$ for all $i=1\ldots S$.
	\end{lemma}
	\begin{proof} 
From $u_i \in L^{q_0}(Q_T)$ for all $i=1,\ldots, S$, we have $f_i(u)\in L^{q_0/\nu}(Q_T)$. Moreover note that the quasi-positivity assumption \eqref{P} ensure non-negative solutions $u$ for non-negative 
initial data $u_{i,0}$. Hence, the concentrations $u_i$ satisfy the (non-sign-changing) porous media equation 
		\begin{equation*}
		\partial_t u_i - d_i\Delta(u_i^{m_i}) = f_i(u) \in L^{q_0/\nu}(Q_T).
		\end{equation*}
	Lemma \ref{heat-regularity} implies that if $q_0/\nu \geq \frac{d+2}{2}$, then $u_i\in L^{r}(Q_T)$ for all $r<\infty$, while if $q_0/\nu < \frac{d+2}{2}$, then 
		\begin{equation*}
		u_i \in L^{s}(Q_T) \quad \text{ for all } \quad s < q_1:= \frac{(md+2)q_0}{\nu(d+2)-2q_0} \leq  \frac{(m_id+2)q_0}{\nu(d+2)-2q_0}, \qquad \text{for all } i=1\ldots S,
		\end{equation*}
		since $m\leq m_i$. We then construct a sequence $q_n$ (equally for all $i=1,\ldots,S$) such that
		\begin{equation}\label{iteration-q}
		q_{n+1} = \frac{(md+2)q_n}{\nu(d+2)-2q_n} \quad \text{ for } n\geq 0.
		\end{equation}
		It follows that
		\begin{equation*}
		\frac{q_{n+1}}{q_n} = \frac{md+2}{\nu(d+2)-2q_n}.
		\end{equation*}
		Therefore, as long as $\nu(d+2)-2q_n>0 \iff q_n < \frac{(d+2)\nu}{2}$,
		\begin{equation*}
		\frac{q_{n+1}}{q_n} > 1 \text{ for all } n \geq 0 \quad \Longleftrightarrow \quad q_0 > \frac{d(\nu-m)+2(\nu-1)}{2}.
		\end{equation*}
		Hence with $q_0 > \frac{d(\nu-m)+2(\nu-1)}{2}$, after finitely many steps we arrive at $q_{n} > \frac{(d+2)\nu}{2}$. From $u_i\in L^{s}(Q_T)$ for all $s<q_n$, we have in particular $u_i\in L^{\frac{(d+2)\nu}{2}}(Q_T)$, which implies $f_i(u)\in L^{\frac{d+2}{2}}(Q_T)$ for $i=1,\ldots, S$. 
By applying Lemma \ref{heat-regularity} once more we obtain $u_i \in L^r(Q_T)\cap L^{\infty}(0,T;L^q(\Omega))$ for all $r, q < \infty$. Thus,
		\begin{equation*}
		\partial_t u_i - d_i\Delta(u_i^{m_i}) = f_i(u)\in L^{\infty}(0,T;L^s(\Omega)) \quad \text{ for all } s<\infty
		\end{equation*}
		with $\|f_i(u)\|_{L^{\infty}(0,T;L^s(\Omega))} \leq C_T$. By considering $s$ large enough such that
		\begin{equation*}
		\frac{d}{d(m_i - 1) + 2s} < 1
		\end{equation*}
		holds, we have 
		\begin{equation*}
		\|u_i\|_{L^\infty(Q_T)} \leq C_T \quad \text{ for all } \quad i=1,\ldots, S,
		\end{equation*}
		thanks to the following Lemma \ref{Linf}.
	\end{proof}
	\begin{lemma}\label{Linf}
		Let $m\geq 1$. Let $u$ be the solution to the no-flux porous media equation
		\begin{equation*}
		u_t - \delta\Delta(|u|^{m-1}u) = f, \qquad  \nu\cdot\nabla (|u|^{m-1}u) = 0, \qquad u(\cdot, 0) = u_0
		\end{equation*}
		with $u_0\in L^{\infty}(\Omega)$. If $\|f\|_{L^{\infty}(0,T;L^p(\Omega))} \leq C_T$ with $p$ large enough such that $d/(d(m-1)+2p) < 1$, then
		\begin{equation*}
		\|u\|_{L^{\infty}(Q_T)} \leq C_T,
		\end{equation*}
		where the constant $C_T$ depends polynomially on $T$.
	\end{lemma}
	\begin{remark}
This regularity is well known for porous medium equation. However, since we were unable to find a reference, which includes the polynomial grow of the constant $C_T$, we provide in the following a proof for the sake of completeness.
	\end{remark}
	\begin{proof}
		Let $S(t)$ be the nonlinear semigroup corresponding to the homogeneous equation $u_t - \delta\Delta(|u|^{m-1}u) = 0$ with homogeneous Neumann boundary condition $\nabla(|u|^{m-1}u)\cdot \overrightarrow{n} = 0$. Then, we have the following well known $L^p-L^\infty$ estimate (see e.g. \cite[Theorem 3.2]{GM13}) 
		\begin{equation*}
		\|S(t)u_0\|_{L^{\infty}(\Omega)} \leq C(\|u_0\|_{L^p(\Omega)}+\|u_0\|_{L^p(\Omega)}^{\sigma_p} t^{-\alpha_p})
		\end{equation*}
		where
		\begin{equation*}
		\sigma_p = \frac{2p}{d(m-1) + 2p} \quad \text{ and } \quad \alpha_p= \frac{d}{d(m-1)+2p}.
		\end{equation*}
		Moreover, the semi-group propagates the $L^{\infty}$-norm, i.e. 
		\begin{equation*}
		\|S(t)u_0\|_{L^{\infty}(\Omega)} \leq \|u_0\|_{L^{\infty}(\Omega)}.
		\end{equation*}
		Hence, we estimate the solution of the inhomogeneous equation by using Duhamel's formula for all $0\leq t \leq T$,
		%\begin{equation*}
		%u(t) = S(t)u_0 + \int_0^tS(t-s)f(s)ds
		%\end{equation*}
		%we have
		\begin{equation*}
		\begin{aligned}
		\|u(t)\|_{L^{\infty}(\Omega)} &\leq \|S(t)u_0\|_{L^{\infty}(\Omega)} + \int_0^t\|S(t-s)f(s)\|_{L^{\infty}(\Omega)}ds\\
		&\leq \|u_0\|_{L^{\infty}(\Omega)} + \int_0^tC\left(\|f(s)\|_{L^p(\Omega)}+\|f(s)\|_{L^p(\Omega)}^{\sigma_p}(t-s)^{-\alpha_p}\right)ds\\
		&\leq \|u_0\|_{L^{\infty}(\Omega)} + C T \|f\|_{L^\infty(0,T;L^p(\Omega))} + C \|f\|_{L^\infty(0,T;L^p(\Omega))}^{\sigma_p} \frac{T^{1-\alpha_p}}{1-\alpha_p},
		\end{aligned}
		\end{equation*}
		since $\alpha_p <1$ by the assumption on $p$ guarantees the convergence of the last integral on the right hand side. This finishes the proof.
	\end{proof}
	%\begin{remark}
	%	For Lemma \ref{lem:S}, with the help of Lemma \ref{duality}, we see that if $m\geq 2$ then weak solutions to \eqref{S} are always regular.
	%%	In Lemma \ref{S} we see that if $m\geq 2$ then we only need $q_0>1$ to obtain regular solutions. And $u_i \in L^{q_0}(Q_T)$ with $q_0>0$ is guaranteed 
	%\end{remark}
	
%	\begin{theorem}[Regularity of solutions to \eqref{S}]\label{theo:main}
%		If the nonlinear diffusion exponents $m_i$ satisfy
%		\begin{equation*}
%		m = \min_{i=1\ldots S}\{m_i\} > \nu - \min\left\{\frac{4}{2+d}; 1\right\}
%		\end{equation*}
%		then any weak solution to \eqref{S} is bounded and satisfies the estimate with $i=1,\ldots, S$
%		\begin{equation*}
%		\|u_i\|_{L^{\infty}(Q_T)} \leq C_T \quad \text{ for all } \quad T>0.
%		\end{equation*}
%	\end{theorem}
%	\begin{remark}
%		We remark that in two dimensional case $N=2$, if $m_i=1$ for $i=1,\ldots, S$, i.e. linear diffusion, then \eqref{S} also possesses global strong solution thanks to a slightly improved duality lemma, see \cite[Lemma 3.2]{CDF14}.
%	\end{remark}
	Now we are ready to prove the boundedness of solutions to \eqref{S}:
	\begin{proof}[Proof of Theorem \ref{T1.2}]
		Assuming $m_i > \nu - 1$, the existence of weak solutions follows similar to \cite{LP17,PR16} and is proven in Section \ref{app} in detail. By the duality estimates in Lemma \ref{duality}, we have
		\begin{equation*}
		u_i \in L^{m_i+1}(Q_T) \quad \text{ for all}\quad i=1,\ldots, S.
		\end{equation*}
		Because $m_i > \nu - \frac{4}{2+d}$ it follows that
		\begin{equation*}
		m_i+1 > \frac{d(\nu-m_i)+2(\nu-1)}{2}.
		\end{equation*}
		Therefore, Lemma \ref{lem:S} yields $u_i \in L^{\infty}(Q_T)$ and $\|u_i\|_{L^{\infty}(Q_T)} \leq C_T$ for arbitrary $T>0$, which shows that the weak solutions are bounded and the $L^{\infty}(\Omega)$ norms grows at most polynomially in time.
		
		The local H\"older continuity of the bounded weak solutions is a classical result, see e.g. \cite{DF85} or \cite[Theorem 7.17]{Vaz07}.
	\end{proof}
	
	\section{Convergence to equilibrium}\label{sec:conv}
In this section, we prove exponential convergence to equilibrium of solutions to \eqref{R} by using the entropy method.
We start by recalling the entropy (free energy) functional
\begin{equation*}
E[a,b] = \sum_{i=1}^{M}\int_{\Omega}(a_i\ln{a_i} - a_i + 1)dx
	+
\sum_{j=1}^{N}\int_{\Omega}(b_j\ln{b_j} - b_j + 1)dx
\end{equation*}
and its non-negative entropy production (free energy dissipation) functional $D[a,b] := -\frac{d}{dt}E[a,b]$, i.e.
\begin{equation*}
D[a,b] 	= 
\sum_{i=1}^{M}d_i\int_{\Omega}\frac{|\nabla a_i|^2}{a_i^{2-m_i}}dx 
+	\sum_{j=1}^{N}h_j\int_{\Omega}\frac{|\nabla b_j|^2}{b_j^{2-p_j}}dx
+ \int_{\Omega}(a^\alpha- b^\beta)\ln{\frac{a^\alpha}{b^\beta}}dx\ge 0,
\end{equation*}
where we have used the short hand notation
\begin{equation*}
a^\alpha=\prod_{i=1}^M a_i^{\alpha_i}
\quad
\text{and}
\quad
b^\beta=\prod_{j=1}^N b_j^{\beta_j}.
\end{equation*}
Moreover, the following additivity property of the relative entropy holds
	\begin{align*}
	E[a,b]-E[a_\infty,b_\infty]
	&=
	\sum_{i=1}^M \int_\Omega \left(a_i\ln\frac{a_i}{a_{i\infty}}-a_i+a_{i\infty}\right)dx
	+  
	\sum_{j=1}^N\int_\Omega \left(b_j\ln\frac{b_j}{b_{j\infty}}-b_j+b_{j\infty}\right)dx\\
	&=
	\sum_{i=1}^M \int_\Omega \left(a_i\ln\frac{a_i}{\overline{a_{i}}}\right)dx
	+  
	\sum_{j=1}^N\int_\Omega \left(b_j\ln\frac{b_j}{\overline{b_{j}}} \right)dx\\
	&\quad +
	\sum_{i=1}^M \int_\Omega \left(\overline{a_i}\ln\frac{\overline{a_i}}{a_{i\infty}}-\overline{a_i}+a_{i\infty}\right)dx
	+  
	\sum_{j=1}^N\int_\Omega \left(\overline{b_j}\ln\frac{\overline{b_j}}{b_{j\infty}}-\overline{b_j}+b_{j\infty}\right)dx.
	\end{align*}
\medskip

{The first Lemma \ref{GenLogSob} of this section states the generalisation of the Logarithmic Sobolev Inequality, 
which shall use in our approach.}

% The first main goal of this section is to prepare for the proof of the entropy-entropy production inequality 
%\eqref{EEP} in Section \ref{sec:add}, from which Theorem \ref{T1.3} follows via a Gronwall arguments. }

	\begin{lemma}[A generalised Logarithmic Sobolev Inequalities, \cite{MM17}]\label{GenLogSob}\hfill\\
		Assume that $m \geq (d-2)_+/d$ where $(d-2)_+ = \max\{0, d-2\}$. Then, there exists a constant $C(\Omega,m) >0$ such that
		\begin{equation*}
		\int_{\Omega}\frac{|\nabla u|^2}{u^{2-m}}dx \geq C(\Omega,m)\,\overline{u}^{\,m-1}\int_{\Omega}u\ln{\frac{u}{\overline{u}}}dx \geq C(\Omega,m)\,\overline{u}^{\,m-1}\|\sqrt{u} - \overline{\sqrt{u}}\|^2
		\end{equation*}
		where $\overline{u}=\int_\Omega udx$.
	\end{lemma}
	\begin{proof}
The first inequality follows from \cite{MM17}. The second estimate follows from an elementary inequality:
		\[
		\int_\Omega u\ln\frac{u}{\overline{u}}dx
		=
		\int_\Omega (u\ln\frac{u}{\overline{u}}-u+\overline{u})dx
		\geq
		\int_\Omega(\sqrt{u}-\sqrt{\overline{u}})^2dx.
		\]
	\end{proof}
	%\begin{lemma}\label{additivity} The following identity holds
	%	\begin{equation*}
	%		\begin{aligned}
	%			E[u(t)] - E[u_{\infty}] &= (E[u(t)] - E[\overline{u}(t)]) + (E[\overline{u}(t)] - E[u_{\infty}])\\
	%			&= \sum_{i=1}^{4}\int_{\Omega}u_i\ln{\frac{u_i}{\overline{u}_i}}dx + \sum_{i=1}^{4}\left(\overline{u}_i\ln{\frac{\overline{u}_i}{u_{i,\infty}}} - \overline{u}_i + u_{i,\infty}\right).
	%		\end{aligned}
	%	\end{equation*}
	%\end{lemma}
	
The estimates in Lemma \ref{GenLogSob} constitute a  generalisation of the Logarithmic Sobolev Inequality \eqref{LSI}, which is recovered by setting $m = 1$ and for which the pre-factor $\overline{u}^{m - 1}$ vanishes. 
In the case of porous media diffusion $m > 1$, the pre-factor $\overline{u}^{m - 1}$ 
causes the lower bounds in Lemma \ref{GenLogSob} to degenerate for small  
spatial averages $\overline{u}$. In particular, we have by Lemma \ref{GenLogSob} the following lower bound for the entropy production 
	\begin{align}\label{lower-D}
	D[a,b]
	&\geq 
	\sum_{i=1}^M d_iC(\Omega,m_i)\overline{a_i}^{\,m_i-1}\int_\Omega a_i\ln\frac{a_i}{\overline{a_i}}dx
	+ \sum_{j=1}^N h_jC(\Omega,p_j)\overline{b_j}^{\,p_j-1}\int_\Omega b_j\ln\frac{b_j}{\overline{b_j}}dx
	+\int_\Omega (a^\alpha-b^\beta)\ln\frac{a^\alpha}{b^\beta}dx
	\\
	&\geq 
	C_0\left[
	\sum_{i=1}^M \overline{a_i}^{\,m_i-1}\int_\Omega a_i\ln\frac{a_i}{\overline{a_i}}dx
	+ \sum_{j=1}^N \overline{b_j}^{\,p_j-1}\int_\Omega b_j\ln\frac{b_j}{\overline{b_j}}dx
	+\int_\Omega (a^\alpha-b^\beta)\ln\frac{a^\alpha}{b^\beta}dx
	\right]\nonumber
	\end{align}

%
%For the estimate of the entropy dissipation $D[a,b]$, we notice that if $\overline{a_i}$ and $\overline{b_j}$ are bounded below by a positive constant, then the first parenthesis of the above calculation can be controlled by $D[a,b]$. 
	
The problem of degeneracy appears when some averages $\overline{a_i}$ or $\overline{b_j}$ do not satisfy a positive lower bound. 
To overcome this problem, we first observe that due to the mass conservation laws \eqref{conservation-laws} not all spatial averages can be small at the same time. 
If, for instance, a particular $\overline{a}_i$ is sufficiently small (w.r.t. $M_{ij}$) then another $\overline{b}_j$
can't be arbitrarily small because of a mass conservation law \eqref{conservation-laws} connecting these two species, i.e. 
\begin{equation}\label{mass-conv}
\beta_j\overline{a_i}+\alpha_i\overline{b_j}=M_{ij}>0,
\end{equation}

The following crucial Lemma \ref{IndDiff} shows functional inequalities, which quantity the so-called "indirect diffusion effect"
and allows to compensate the lacking lower bounds for the species, whose spatial averages do not satisfy a lower bound. 

We first introduce some convenient notations: 
%Therefore when the average $\overline{u}_i$ is bounded below by a positive constant then we can reuse the usual logarithmic Sobolev inequality. The difficulty appears when the average is small. 
	\begin{align*}
	&A_i=\sqrt{a_i},\ A_{i\infty}=\sqrt{a_{i\infty}},
	&B_j=\sqrt{b_j},\ B_{j\infty}=\sqrt{b_{j\infty}},\qquad\\ 
	&\delta_i(x)=A_i(x)-\overline{A_i}, \quad \forall x\in\Omega, &\eta_j(x)=B_j(x)-\overline{B_j},\quad \forall x\in\Omega,
	\end{align*}
	where
	\[
	\overline{A_i}=\int_\Omega A_idx \quad \text{and} \quad 
	\overline{B_j}=\int_\Omega B_jdx.
	\]
	Moreover,
	\[
	A^\alpha=\prod^M_{i=1}A_i^{\alpha_i}\quad \text{and}\quad 
	B^\beta=\prod^N_{j=1}B_j^{\beta_j}.
	\]
	The conservation laws are now rewritten as
	\begin{equation}\label{lawsAB}
		\beta_j \overline{A_i^2} + \alpha_i\overline{B_j^2} = M_{ij} > 0 \qquad \forall i=1\ldots M, j=1\ldots N.
	\end{equation}
	
	\begin{lemma}[``Indirect diffusion transfer" functional inequality]\label{IndDiff}\hfill\\
		Let $A_i, B_j:\Omega \to \mathbb R_+$ with $i=1\ldots M$ and $j=1\ldots N$ be nonnegative functions satisfying the conservation laws \eqref{lawsAB} and $\varepsilon>0$	be a constant to be determined later. 
		Assume that for some $J\in\{1,\ldots,N\}$,
		\[
		\overline{B_j^2}\leq\varepsilon \quad \text{ for all } \quad j=1\ldots J.
		\]
		Then, there exists a constant $K_1$ which depends on $\varepsilon$ such that:
		\begin{equation}\label{star}
		\sum_{i=1}^M\|\delta_i\|^2
		+\sum_{j=J+1}^N\|\eta_j\|^2
		+\|A^\alpha-B^\beta\|^2\geq K_1\sum_{j=1}^J\|\eta_j\|^2
		\end{equation}
	\end{lemma}
	\begin{remark}
		Note that when the last term on the left hand side $\|A^\alpha - B^\beta\|^2$ diverges, the inequality holds trivially. Therefore, in the proof we only consider the case when it is finite.
	\end{remark}
	\begin{proof}
		Due to the mass conservation laws \eqref{lawsAB}, we have the following natural bounds,
		\[
		\overline{A_i^2},\overline{B_j^2}\leq M_0^2,\qquad \forall i=1,\ldots,M,\ \forall j=1,\ldots,N
		\]	
		for some constant $M_0>0$. Therefore, by Jensen's inequality, recalling that $|\Omega| = 1$,
		\[
		\overline{A_i}\leq\sqrt{\overline{A_i^2}}\leq M_0,\qquad 
		\overline{B_j}\leq\sqrt{\overline{B_j^2}}\leq M_0,\quad \forall i,j.
		\]
		From these bounds we get an upper bound for the right hand side of \eqref{star}
		\[
		\sum_{j=1}^J\|\eta_j\|^2=
		\sum_{j=1}^J(\overline{B_j^2}-\overline{B_j}^2)
		\leq 
		\sum_{j=1}^J\overline{B_j^2}\leq M_0^2J.
		\]
		We consider the following two cases.
		
		{\bf Case 1}:  If there exists $i\in \{1,\ldots,M\}$ such that  $\|\delta_i\|^2\geq\varepsilon$ or there exists a $j\in \{J+1,\ldots,N\}$ such that $\|\eta_j\|^2\geq\varepsilon$, we have:
		\[
		\sum_{i=1}^M\|\delta_i\|^2
		+\sum_{j=J+1}^N\|\eta_j\|^2
		+\|A^\alpha-B^\beta\|^2
		\geq\varepsilon\geq \frac{\varepsilon}{M_0^2J}\sum_{j=1}^J\|\eta_j\|^2
		\] 
		hence, the desired inequality \eqref{star} holds with $K_1=\frac{\varepsilon}{M_0^2J}$.
		
		{\bf Case 2}: Assume $\|\delta_i\|^2\leq\varepsilon$ for all $i\in \{1,\ldots,M\}$ and  $\|\eta_j\|^2\leq\varepsilon$ for all $j\in \{J+1,\ldots,N\}$, which together with the above assumption $\overline{B_j^2}\leq\varepsilon$ and $\overline{\eta_j^2} \leq \overline{B_j^2}$ for all $j=1\ldots J$ implies $\|\eta_j\|^2\leq\varepsilon$ for all  $j\in \{1,\ldots,N\}$., 
		
		Let $\lambda>0$ and denote by
		\[
		\Omega_{iA}=\{
		x\in\Omega:|\delta_i(x)|\leq\lambda\sqrt{\varepsilon}
		\}\quad \text{for}\ i=1,\dots,M.
		\]
		Then 
		\[
		\varepsilon\geq\int_\Omega |\delta_i(x)|^2dx\geq
		\int_{\Omega\backslash\Omega_{iA}} |\delta_i(x)|^2dx\geq
		\lambda^2\varepsilon|\Omega\backslash\Omega_{iA}|
		\]
		thus 
		\[
		|\Omega\backslash\Omega_{iA}|\leq\frac{1}{\lambda^2} \quad \text{ which implies} \quad
		|\Omega_{iA}|\geq1-\frac{1}{\lambda^2}
		\]
		Similarly we get,
		\[
		|\Omega_{jB}|\geq1-\frac{1}{\lambda^2} \quad \text{ where } \Omega_{jB} = \{x\in\Omega:|\eta_j(x)| \leq \lambda\sqrt{\varepsilon}\} \quad\forall j=1,\ldots,N.
		\]
		Now choose $\lambda^2=2(M+N)$ and consider $G=\cap_{i=1}^M\Omega_{iA}\cap^N_{j=1}\Omega_{jB}$. Then,  we have $|G|\geq\frac12$.
		Note that $|\delta_i(x)|\leq\lambda\sqrt{\varepsilon}$ and $|\eta_j(x)|\leq\lambda\sqrt{\varepsilon}$ for all $x\in G$ and for all $i,j$. Moreover, $\forall x\in G$
		\[
		A_i(x)=\overline{A_i}+\delta_i(x)
		\leq \overline{A_i}+|\delta_i(x)|\leq
		M_0+\lambda\sqrt{\varepsilon}\leq2M_0
		\]
		and similarly $B_j(x)\leq2M_0,\ \forall i,j$ if we choose $\varepsilon$ such that 
		\[
		\lambda\sqrt{\varepsilon}\leq M_0.
		\]
		By Taylor's expansion, we have
		\begin{align*}
		A^\alpha =\prod^M_{i=1}A_i^{\alpha_i}	=\prod^M_{i=1}(\overline{A_i}+\delta_i)^{\alpha_i}	
		=\prod^M_{i=1}\overline{A_i}^{\alpha_i}	+R(\overline{A_i},\delta_i)\sum^M_{i=1}\delta_i
		\end{align*}
		where the remainder terms $R$ depends polynomially on $\overline{A_i}$ and $\delta_i$. Note that $|R(\overline{A}_i, \delta_i)| \leq C_0(M_0)$ on $G$, we estimate with $(x-y)^2\geq\frac12x^2-y^2$
		\begin{align*}
		\|A^\alpha-B^\beta\|^2&=\int_\Omega \left(
		\prod_{i=1}^MA_i^{\alpha_i}-B^\beta
		\right)^2dx	
		\\ &\geq
		\int_G \left(
		\prod_{i=1}^M\overline{A_i}^{\alpha_i}
		-B^\beta
		+R(\overline{A_i},\delta_i)\sum_{i=1}^M\delta_i
		\right)^2dx	
		\\ &\geq\frac12
		\int_G \left(
		\prod_{i=1}^M\overline{A_i}^{\alpha_i}
		-B^\beta
		\right)^2dx	
		-\int_G |R(\overline{A_i},\delta_i)|^2|\sum_{i=1}^M\delta_i|^2
		\\ &\geq
		\frac12\int_G \left(
		\prod_{i=1}^M\overline{A_i}^{\alpha_i}-B^\beta
		\right)^2dx	
		-C_0(M_0)^2M\int_G\sum_{i=1}^M|\delta_i|^2
		\\ &\geq
		\frac12\int_G \left(
		\prod_{i=1}^M\overline{A_i}^{\alpha_i}-B^\beta
		\right)^2dx	
		-C_0(M_0)^2M\int_G\sum_{i=1}^M\|\delta_i\|^2
		\\ &\geq
		\frac12\int_G \left(
		\prod_{i=1}^M\overline{A_i}^{\alpha_i}-B^\beta
		\right)^2dx	
		-C_0(M_0)^2M^2\varepsilon
		\end{align*}
		where we used $\|\delta_i\|^2 \leq \varepsilon$ in the last inequality.
		
		In order to estimate further, we use again Taylor's expansion
		\[
		B^\beta=
		\prod_{j=1}^N(\overline{B_j}+\eta_j)^{\beta_j}
		=
		\prod_{j=1}^N\overline{B_j}^{\beta_j}
		+Q(\overline{B_j},\eta_j)\sum_{j=1}^N\eta_j
		\]
		where again, $Q$ depends polynomially on $\overline{B_j},\eta_j$, which implies $|Q(\overline{B}_j, \eta_j)| \leq C_1(M_0)$ on $G$. Therefore,
		\begin{align*}
		\int_G\left(
		\prod_{i=1}^M\overline{A_i}^{\alpha_i}
		-B^\beta
		\right)^2	dx&=
		\int_G\left(
		\prod_{i=1}^M\overline{A_i}^{\alpha_i}
		-\prod_{j=1}^N\overline{B_j}^{\beta_j}
		-Q(\overline{B_j},\eta_j)\sum_{j=1}^N\eta_j
		\right)^2	dx
		\\&\geq
		\frac12\int_G\left(
		\prod_{i=1}^M\overline{A_i}^{\alpha_i}
		-\prod_{j=1}^N\overline{B_j}^{\beta_j}\right)^2	dx
		-\int_G|Q(\overline{B_j},\eta_j)|^2|\sum_{j=1}^N\eta_j
		|^2	dx
		\\&\geq
		\frac12\int_G\left(
		\prod_{i=1}^M\overline{A_i}^{\alpha_i}
		-\prod_{j=1}^N\overline{B_j}^{\beta_j}\right)^2	dx
		-C_1(M_0)^2N^2\varepsilon
		\end{align*}
		where we used that  $\|\eta_j\|^2\leq\varepsilon$ for all $j=1,\ldots, N$.
		
		Combining these two estimates, we arrive at
		\begin{equation}\label{twostar}
		\|A^\alpha-B^\beta\|^2\geq\frac14|G|\left(
		\prod_{i=1}^M\overline{A_i}^{\alpha_i}
		-\prod_{j=1}^N\overline{B_j}^{\beta_j}\right)^2
		-\varepsilon\left(\frac12C_1(M_0)^2N^2+C_0(M_0)^2M^2\right).
		\end{equation}
		By Jensen's inequality and the assumption of the Lemma, we have
		\[
		\overline{B_j}\leq\sqrt{\overline{B_j^2}}\leq\sqrt{\varepsilon}
		,\quad \forall j=1,\ldots,J.
		\]
		On the other hand $\overline{B_j}\leq\sqrt{\overline{B_j^2}}\leq M_0
		,\ \forall j=J+1,\ldots,N$. Thus, the conservation law  \eqref{lawsAB} and $\|\delta_i\|^2 \leq \varepsilon$ yield
		\begin{equation*}
		\overline{A}_i = \sqrt{\overline{A_i^2} - \|\delta_i\|^2} = \sqrt{\frac{1}{\beta_1}(M_{i1}-\alpha_i\overline{B_1^2}) - \|\delta_i\|^2} \geq \sqrt{\frac{M_{i1}}{\beta_1} - \frac{\alpha_i}{\beta_1}\varepsilon - \varepsilon} \quad \forall i=1,\ldots, M.
		\end{equation*}
		Hence, by using $|G| \geq \frac 12$ we get from \eqref{twostar} that
		\begin{equation*}
		\|A^\alpha - B^\beta\|^2 \geq \frac 18\left[\prod_{i=1}^M\left(\frac{M_{i1}}{\beta_1} - \frac{\alpha_i}{\beta_1}\varepsilon - \varepsilon\right)^{\alpha_i/2} - \prod_{j=1}^{J}(\sqrt \varepsilon)^{\beta_j}\prod_{j=J+1}^{N}M_0^{\beta_j} \right]^2 - C_2\varepsilon.
		\end{equation*}
		Because the right hand side of the above inequality converges to $\frac 18\prod_{i=1}^M\bigl(\frac{M_{i1}}{\beta_1}\bigr)^{\alpha_i}$ as $\varepsilon\to 0$, we can choose $\varepsilon>0$ small enough, but still explicit, such that 
		\begin{equation*}
		\|A^\alpha - B^\beta\|^2 \geq \frac{1}{16}\prod_{i=1}^M\Bigl(\frac{M_{i1}}{\beta_1}\Bigr)^{\alpha_i} \geq \frac{1}{16M_0^2J}\prod_{i=1}^{M}\Bigl(\frac{M_{i1}}{\beta_1}\Bigr)^{\alpha_i}\sum_{j=1}^{J}\|\eta_j\|^2,
		\end{equation*}
		which implies the desired inequality \eqref{star} with the constant
		\begin{equation*}
		K_1 = \frac{1}{16M_0^2J}\prod_{i=1}^M\Bigl(\frac{M_{i1}}{\beta_1}\Bigr)^{\alpha_i}.
		\end{equation*}
	\end{proof}
	
	\begin{lemma}[An time-dependent entropy-entropy production estimate]\label{variant-eede}\hfill\\
		Let $(a,b) = (a_1,\ldots, a_M, b_1,\ldots, b_N)$ with $a_i, b_j: Q_T \to \mathbb R_+$ be nonnegative functions, which satisfy the conservation laws \eqref{conservation-laws}. Moreover,
		\begin{equation*}
		\|a_i\|_{L^{\infty}(Q_T)} \leq C_T \quad \text{ and } \quad \|b_j\|_{L^\infty(Q_T)} \leq C_T \quad \text{ for all } i, j.
		\end{equation*}
		Then, there exists a constant $K_2>0$ such that for all $T>0$,
		\begin{equation*}
		D[a(T),b(T)] \geq K_2 \frac{1}{1+\ln(1+T)} (E[a(T), b(T)] - E[\ainf,\binf]).
		\end{equation*}
	\end{lemma}
	\begin{proof}
		Let $\varepsilon>0$ be a small constant chosen in Lemma \ref{IndDiff}. We will consider two cases and for convenience we will drop $T$ in $a_i(T)$ and $b_j(T)$ when there is no confusion.
		
\noindent{\bf Case 1.} Assume $\overline{a}_i \geq \varepsilon$ for all $i=1,\ldots, M$ and $\overline{b}_j \geq \varepsilon$ for all $j=1,\ldots, N$. By applying Lemma \ref{GenLogSob}, we have
		\begin{equation*}
		\begin{aligned}
		D[a,b] &\geq \sum_{i=1}^{M}d_iC(\Omega,m_i)\varepsilon^{m_i-1}\int_{\Omega}a_i\ln{\frac{a_i}{\overline{a}_i}}dx + \sum_{j=1}^{N}h_jC(\Omega,p_j)\varepsilon^{p_j-1}\int_{\Omega}b_j\ln{\frac{b_j}{\overline{b}_j}}dx + \int_{\Omega}(a^\alpha - b^\beta)\ln\frac{a^\alpha}{b^\beta}dx\\
		&\geq K_3\left[\sum_{i=1}^{M}\int_{\Omega}a_i\ln{\frac{a_i}{\overline{a}_i}}dx + \sum_{j=1}^{N}\int_{\Omega}b_j\ln{\frac{b_j}{\overline{b}_j}}dx + \int_{\Omega}(a^\alpha - b^\beta)\ln\frac{a^\alpha}{b^\beta}dx\right]
		\end{aligned}
		\end{equation*}
		with
		\begin{equation*}
		K_3 = \min_{i=1\ldots M; j=1\ldots N }\{d_iC(\Omega, m_i)\varepsilon^{m_i-1}; h_jC(\Omega,p_j)\varepsilon^{p_j-1}; 1\}.
		\end{equation*}
		Using an entropy-entropy production inequality in case of system \eqref{R} with linear diffusion, see Lemma \ref{EnEnDiEs} below, we know that
		\begin{equation*}
		\sum_{i=1}^{M}\int_{\Omega}a_i\ln{\frac{a_i}{\overline{a}_i}}dx + \sum_{j=1}^{N}\int_{\Omega}b_j\ln{\frac{b_j}{\overline{b}_j}}dx + \int_{\Omega}(a^\alpha - b^\beta)\ln\frac{a^\alpha}{b^\beta}dx \geq K_4(E[a,b] - E[\ainf,\binf])
		\end{equation*}
		for an explicit constant $K_4>0$. Therefore,
		\begin{equation*}
		D[a,b] \geq K_3K_4(E[a,b] - E[\ainf,\binf]).
		\end{equation*}
		
		\medskip
\noindent{\bf Case 2.} Suppose either $\overline{a}_i \leq \varepsilon$ for some $i\in \{1,\ldots, M\}$ or $\overline{b}_j \leq \varepsilon$ for some $j=1,\ldots, N$.
		
		Due to the mass conservation laws $\beta_j \overline{a}_i + \alpha_i\overline{b}_j = M_{ij}$, it cannot happen that $\overline{a}_i \leq \varepsilon$ and $\overline{b}_j \leq \varepsilon$ simultaneously for a sufficiently small $\varepsilon$, e.g. $\varepsilon < \frac{M_{ij}}{2}\min\left\{\frac{1}{\beta_j};\frac{1}{\alpha_i} \right\}$. Therefore, without loss of generality, we can assume that 
		\begin{equation*}
		\overline{b}_j \leq \varepsilon\quad \forall j=1,\ldots, J \qquad \text{ and } \qquad \overline{b}_j \geq \varepsilon\quad \forall j=J+1,\ldots,N
		\end{equation*}
		for some $J \in \{1,\ldots, N\}$. Moreover, by mass conservation laws
		\begin{equation*}
		\overline{a}_i = \frac{1}{\beta_1}(M_{i1}-\alpha_i\overline{b}_1) \geq \frac{1}{\beta_1}(M_{i1} - \alpha_i\varepsilon), \qquad \text{for all } i=1,\ldots,M.
		\end{equation*}
		Thus, we can apply Lemma \ref{GenLogSob} to $D[a,b]$ and estimate
		\begin{equation*}
		\begin{aligned}
		D[a,b] &\geq \sum_{i=1}^{M}d_iC(\Omega,m_i)\left[\frac{1}{\beta_1}(M_{i1}-\alpha_i\varepsilon)\right]^{m_i-1}\int_{\Omega}a_i\ln\frac{a_i}{\overline{a}_i}dx\\
		&\qquad + \sum_{j=J+1}^{N}h_jC(\Omega,p_j)\varepsilon^{p_j-1}\int_{\Omega}b_j\ln\frac{b_j}{\overline{b}_j}dx + \int_{\Omega}(a^\alpha - b^\beta)\ln\frac{a^\alpha}{b^\beta}dx\\
		& \geq K_5\left[\sum_{i=1}^M\|\sqrt{a_i} - \overline{\sqrt{a_i}}\|^2 + \sum_{j=J+1}^N\|\sqrt{b_j} - \overline{\sqrt{b_j}}\|^2 + \|A^\alpha - B^\beta\|^2\right]\\
		&= K_5\left[\sum_{i=1}^M\|\delta_i\|^2 + \sum_{j=J+1}^N\|\eta_j\|^2 + \|A^\alpha - B^\beta\|^2\right],
		\end{aligned}
		\end{equation*}
		where we have used $(x-y)\ln(x/y) \geq 4(\sqrt x - \sqrt y)^2$ and 
		\begin{equation*}
		K_5 = \min_{i=1\ldots M; j=J+1\ldots N}\left\{d_iC(\Omega,m_i)\left[\frac{1}{\beta_1}(M_{i1}-\alpha_i\varepsilon)\right]^{m_i-1}; h_jC(\Omega,p_j)\varepsilon^{p_j-1}; 4 \right\}.
		\end{equation*}
		Applying Lemma \ref{IndDiff} yields
		\begin{equation*}
		D[a,b] \geq K_6\left[\sum_{i=1}^M\|\delta_i\|^2 + \sum_{j=1}^N\|\eta_j\|^2 + \|A^\alpha - B^\beta\|^2\right]
		\end{equation*}
		where
		\begin{equation*}
		K_6  = \frac 12 \min\{K_5; K_5K_1\}.
		\end{equation*}
		By using another functional inequality, which was already proven in the case of linear diffusion, see \eqref{bao3} in Section \ref{sec:add}, we have
		\begin{equation}\label{1star}
		D[a,b]\geq K_7\left[\sum_{i=1}^{M}(\|\delta_i\|^2 + |\sqrt{\overline{A_i^2}} - A_{i,\infty}|^2) + \sum_{j=1}^{N}(\|\eta_j\|^2 + |\sqrt{\overline{B_j^2}} - B_{j,\infty}|^2) \right].
		\end{equation}
		
		Now, we estimate $E[a,b] - E[\ainf,\binf]$ from above. Consider the two variables function
		\begin{equation*}
		\Phi(x,y) = \frac{x\ln(x/y) - x + y}{(\sqrt x- \sqrt y)^2}
		\end{equation*}
		which is continuous in $(0,\infty)^2$ and $\Phi(\cdot, y)$ is increasing for each fixed $y>0$. It holds that
		\begin{equation}\label{upper-E}
		\begin{aligned}
		&E[a,b] - E[\ainf,\binf]\\ &=\sum_{i=1}^{M}\int_{\Omega}\Phi(a_i,a_{i,\infty})(A_i - A_{i,\infty})^2dx + \sum_{j=1}^N\int_{\Omega}\Phi(b_j, b_{j,\infty})(B_j - B_{j,\infty})^2dx\\
		&\leq \max_{i=1\ldots M; j=1\ldots N}\{\Phi(\|a_i\|_{L^{\infty}(Q_T)}, a_{i,\infty}); \Phi(\|b_j\|_{L^{\infty}(Q_T)}, b_{j,\infty})\}\left[\sum_{i=1}^M\|A_i - A_{i,\infty}\|^2 + \sum_{j=1}^N\|B_j - B_{j,\infty}\|^2\right]\\
		&\leq K_8(1+\ln(1+T))\left[\sum_{i=1}^M(\|\delta_i\|^2 + |\overline{A}_i - A_{i,\infty}|^2) + \sum_{j=1}^{N}(\|\eta_j\|^2 + |\overline{B}_j - B_{j,\infty}|^2) \right],
		\end{aligned}
		\end{equation}
		where in the last inequality we have used the estimates $\|a_i\|_{L^{\infty}(Q_T)} \leq C_T$ and $\|b_j\|_{L^{\infty}(Q_T)} \leq C_T$ and that $C_T$ is a constant growing at most polynomially w.r.t. $T$. 
		
		Next, from $\|\delta_i\|^2 = \overline{A_i^2} - \overline{A}_i^2 = (\sqrt{\overline{A_i^2}} - \overline{A}_i)(\sqrt{\overline{A_i^2}} + \overline{A}_i)$, we have
		\begin{equation*}
		\overline{A}_i = \sqrt{\overline{A_i^2}} - \frac{\|\delta_i\|^2}{\sqrt{\overline{A_i^2}} + \overline{A}_i} = \sqrt{\overline{A_i^2}} - Q_i(A_i)\|\delta_i\| \quad \text{ with } \quad Q_i(A_i) = \frac{\|\delta_i\|}{\sqrt{\overline{A_i^2}} + \overline{A}_i}.
		\end{equation*}
		It's obvious that $Q(A_i) \geq 0$ and moreover
		\begin{equation*}
		Q_i(A_i)^2 = \frac{\overline{A_i^2} - \overline{A}_i^2}{(\sqrt{\overline{A_i^2}} + \overline{A}_i)^2} = \frac{\sqrt{\overline{A_i^2}} - \overline{A}_i}{\sqrt{\overline{A_i^2}} + \overline{A}_i} \leq 1.
		\end{equation*}
		Therefore,
		\begin{equation*}
		\begin{aligned}
		|\overline{A}_i - A_{i,\infty}|^2 &\leq 2\left(|\sqrt{\overline{A_i^2}} - \overline{A}_i|^2 + |\sqrt{\overline{A_i^2}} - A_{i,\infty}|^2\right)\\
		&= 2\left(Q_i(A_i)^2\|\delta_i\|^2 + |\sqrt{\overline{A_i^2}} - A_{i,\infty}|^2\right)\\
		&\leq 2\left(\|\delta_i\|^2 + |\sqrt{\overline{A_i^2}} - A_{i,\infty}|^2\right) \quad \text{ for all } i=1\ldots M
		\end{aligned}
		\end{equation*}
		and similarly 
		\begin{equation*}
		|\overline{B}_j - B_{j,\infty}|^2 \leq 2\left(\|\eta_i\|^2 + |\sqrt{\overline{B_j^2}} - B_{j,\infty}|^2\right) \quad \text{ for all } j=1\ldots N.
		\end{equation*}
		Hence it follows from \eqref{upper-E} that
		\begin{equation}\label{2star}
		E[a,b] - E[\ainf,\binf] \leq 3K_8(1+\ln(1+T))\left[\sum_{i=1}^{M}(\|\delta_i\|^2 + |\sqrt{\overline{A_i^2}} - A_{i,\infty}|^2) + \sum_{j=1}^{N}(\|\eta_j\|^2 + |\sqrt{\overline{B_j^2}} - B_{j,\infty}|^2)\right].
		\end{equation}
		A combination of \eqref{1star} and \eqref{2star} yields
		\begin{equation*}
		D[a,b] \geq \frac{K_7}{3K_8(1+\ln(1+T))}(E[a,b] - E[\ainf,\binf]).
		\end{equation*}
		
		Finally, from {\bf Case 1} and {\bf Case 2}, we can conclude the proof of Lemma \ref{variant-eede} with
		\begin{equation*}
		K_2 = \min\left\{K_3K_4; \frac{K_7}{3K_8}\right\}.
		\end{equation*}
	\end{proof}
\begin{remark}
The assumptions $\|a_i\|_{L^{\infty}(Q_T)} \leq C_T$ and $\|b_j\|_{L^{\infty}(Q_T)} \leq C_T$ in Lemma \ref{variant-eede} are only needed to estimate $E[a,b] - E[\ainf,\binf]$ above as in \eqref{upper-E}. 
In the case of linear diffusion, it is possible to avoid these $L^{\infty}$-bounds by 
using the additivity of the relative entropy (see also the proof of Lemma \ref{EnEnDiEs} in Section \ref{sec:add}), i.e. 
		\begin{equation*}
			E[a,b] - E[\ainf,\binf] = (E[a,b] - E[\overline{a},\overline{b}]) + (E[\overline{a},\overline{b}] - E[\ainf,\binf]).
		\end{equation*}
However, while for linear diffusion, the Logarithmic Sobolev Inequality 
controls to first part $E[a,b] - E[\overline{a},\overline{b}]\le C(C_{\mathrm{LSI}}) D[a,b]$, such an estimate is unclear in the case of porous media diffusion, 
where the generalised Logarithmic Sobolev Inequality in Lemma \ref{GenLogSob} degenerates for states without lower bounds on the spatial 
averages.
\end{remark}
	We need also the following Csisz\'ar-Kullback-Pinsker type inequality. The proof is standard and can be found in e.g. \cite{DFT16,FT17a}.
	\begin{lemma}\label{CKP-ineq}
		There exists a constant $C_{CKP}>0$ such that for any measurable nonnegative functions $a_i, b_j: \Omega \to \mathbb R_+$ satisfying the mass conservation \eqref{mass-conv}, there holds
		\begin{equation*}
			E[a,b] - E[\ainf,\binf] \geq C_{CKP}\left(\sum_{i=1}^M\|a_i - a_{i,\infty}\|_1^2 + \sum_{j=1}^N\|b_j - b_{j,\infty}\|_1^2\right).
		\end{equation*}
	\end{lemma}
	We are ready to prove Theorem \ref{T1.3}.
	\begin{proof}[Proof of Theorem \ref{T1.3}]
		Due to the condition
		\begin{equation*}
			m_i, p_j > \max\left\{\nu - \min\left\{\frac{4}{d+2}; 1\right\}; 1\right\} \qquad \forall i=1\ldots M, j=1\ldots N,
		\end{equation*}
		we can apply Theorem \ref{T1.2} to show boundedness of the weak solution $(a,b)$ to \eqref{R}, i.e.
		\begin{equation*}
			\|a_i\|_{L^{\infty}(Q_T)} \leq C_T, \quad \|b_j\|_{L^{\infty}(Q_T)} \leq C_T, \quad \forall i=1\ldots M, j=1\ldots N.
		\end{equation*}
		By applying Lemma \ref{variant-eede} this yields
		\begin{equation*}
			D[a(T), b(T)] \geq K_2\frac{1}{1+\ln(1+T)} (E[a(T), b(T)] - E[\ainf,\binf]).
		\end{equation*}
		Moreover, due to the boundedness of solutions, we have the entropy-entropy production relation
		\begin{equation*}
			\frac{d}{dt}(E[a,b] - E[\ainf,\binf]) = \frac{d}{dt}E[a,b] = -D[a,b]\leq -K_2\frac{1}{1+\ln(1+T)}(E[a,b] - E[\ainf,\binf]).
		\end{equation*}
		A classical Gronwall's inequality leads to 
		\begin{equation*}
			E[a(T), b(T)] - E[\ainf,\binf] \leq \exp\left(-K_2\int_0^T\frac{d\tau}{1+\ln(1+\tau)}\right)(E[a_0,b_0] - E[\ainf,\binf]).
		\end{equation*}
		By direct calculations
		\begin{equation*}
			\exp\left(-K_2\int_0^T\frac{d\tau}{1+\ln(1+\tau)}\right) \geq \exp\left(-K_2\int_0^T\frac{d\tau}{1+\tau}\right) = (1+T)^{-K_2}.
		\end{equation*}
		Hence,
		\begin{equation}\label{entropy-algebraic}
			E[a(T), b(T)] - E[\ainf,\binf] \leq (1+T)^{-K_2}(E[a_0,b_0] - E[\ainf,\binf]),
		\end{equation}
		and therefore thanks to the Csisz\'ar-Kullback-Pinsker inequality in Lemma \ref{CKP-ineq}
		\begin{equation}\label{algebraic}
			\sum_{i=1}^M\|a_i(T) - a_{i,\infty}\|_1^2 + \sum_{j=1}^N\|b_j(T) - b_{j,\infty}\|_1^2\leq C_{CKP}^{-1}(1+T)^{-K_2}(E[a_0,b_0] - E[\ainf,\binf])
		\end{equation}
		which implies algebraic convergence to equilibrium of solutions to \eqref{R}. 
		
		We will now show that from this it is possible to recover exponential convergence. Since the right hand side of \eqref{algebraic} tends to zero as $T\to \infty$, we can choose 
		\begin{equation}\label{T0}
			T_0 = \max\left\{1;  \left[\frac{C_{CKP}^{-1}(E[a_0,b_0]-E[\ainf,\binf])}{\frac 12\min_{i=1\ldots M;j=1\ldots N}\{a_{i,\infty}^2, b_{j,\infty}^2\}}\right]^{1/K_2} - 1\right\}
		\end{equation}
		which implies for all $t\geq T_0$
		\begin{equation*}
			\|a_{i}(t) - a_{i,\infty}\|_{1} \leq \frac 12 a_{i,\infty} \quad \text{ and } \quad \|b_j(t) - b_{j,\infty}\|_1 \leq \frac 12 b_{j,\infty},
		\end{equation*}
		and thus
		\begin{equation*}
			\overline{a}_i(t) = \|a_{i}(t)\|_1 \geq \frac 12 a_{i,\infty} \quad \text{ and } \quad \overline{b}_j(t) = \|b_j(t)\|_1 \geq \frac 12 b_{j,\infty} \quad \text{ for all } \quad t \geq T_0.
		\end{equation*}
		Therefore, for all $t\geq T_0$, we can apply these lower bounds on the spatial averages bounds and Lemma \ref{GenLogSob} to estimate the entropy-entropy production as follows
		\begin{equation*}
				D[a(t),b(t)] \geq C_1\left[
				\sum_{i=1}^M \int_\Omega a_i\ln\frac{a_i}{\overline{a_i}}dx
				+ \sum_{j=1}^N \int_\Omega b_j\ln\frac{b_j}{\overline{b_j}}dx
				+\int_\Omega (a^\alpha-b^\beta)\ln\frac{a^\alpha}{b^\beta}dx
				\right] \quad \text{ for all } t\geq T_0,
		\end{equation*}
		with
		\begin{equation*}
			C_1 = \min_{i=1\ldots M; j=1\ldots N}\left\{d_iC(\Omega,m_i)\left(\frac{1}{2}a_{i,\infty}\right)^{m_i-1}; h_jC(\Omega,p_j)\left(\frac 12 b_{j,\infty}\right)^{p_j-1}; 1 \right\}.
		\end{equation*}
		By applying again Lemma \ref{EnEnDiEs}, we obtain
		\begin{equation*}
			D[a(t),b(t)] \geq C_1\lambda (E[a(t),b(t)] - E[\ainf,\binf]) \quad \text{ for all } \quad t\geq T_0,
		\end{equation*}
		which in a combination with the classical Gronwall's inequality yields for all $t\geq T_0$,
		\begin{equation*}
			\begin{aligned}
			E[a(t),b(t)] - E[\ainf,\binf] &\leq e^{-\lambda C_1(t-T_0)}(E[a(T_0),b(T_0)] - E[\ainf,\binf])\\
			&\leq e^{-\lambda C_1 t}e^{\lambda C_1 T_0}(1+T_0)^{-K_2}(E[a_0, b_0] - E[\ainf,\binf])\\
			&\leq e^{-\lambda C_1 t}e^{\lambda C_1 T_0}(E[a_0, b_0] - E[\ainf,\binf])
			\end{aligned}
		\end{equation*}
		where we used \eqref{entropy-algebraic} for the second inequality.
		On the other hand, it follows from \eqref{entropy-algebraic} that for all $0\leq t< T_0$,
		\begin{equation*}
			\begin{aligned}
			E[a(t), b(t)] - E[\ainf,\binf] &\leq (1+t)^{-K_2}(E[a_0,b_0] - E[\ainf,\binf])\\
			&\leq e^{-\lambda C_1 t}e^{\lambda C_1T_0}(E[a_0,b_0] - E[\ainf,\binf])			
			\end{aligned}
		\end{equation*}
		Due to the explicitness of $T_0$ in \eqref{T0}, we eventually get the exponential convergence
		\begin{equation*}
			E[a(t),b(t)] - E[\ainf,\binf] \leq C_2e^{-\widehat\lambda t}(E[a_0,b_0] - E[\ainf,\binf]) \quad \text{ for all } \quad t\geq 0,
		\end{equation*}
		{with the constant $C_2=e^{\lambda C_1 T_0}$ and the rate $\widehat\lambda = \lambda C_1$}. Note that $C_2$ is explicit since $T_0$ is explicit (see \eqref{T0}). With another application of the Csisz\'ar-Kullback-Pinsker inequality in Lemma \ref{CKP-ineq}, this yields
		\begin{equation*}
			\sum_{i=1}^{M}\|a_i(t) - a_{i,\infty}\|_1^2 + \sum_{j=1}^N\|b_j(t) - b_{j,\infty}\|_1^2 \leq C_2C_{CKP}^{-1}e^{-\widehat{\lambda}t}(E[a_0,b_0] - E[\ainf,\binf]) \leq C_3e^{-\widehat{\lambda}t}
		\end{equation*}
		with $C_3 = C_2C_{CKP}^{-1}(E[a_0,b_0] - E[\ainf,\binf])$. Finally, by combining the above exponential $L^1$-convergence with the at most polynomial grow $L^{\infty}$ a-priori estimates $\|a_i\|_{L^{\infty}(Q_T)}, \|b_j\|_{L^{\infty}(Q_T)} \leq C_T$, interpolation yields for any $1<p<\infty$,
		\begin{equation*}
			\|a_i(T) - a_{i,\infty}\|_p \leq \|a_i(T) - a_{i,\infty}\|_{\infty}^{\theta}\|a_i(T) - a_{i,\infty}\|_1^{1-\theta}\le C_T^{\theta}C_3^{1-\theta}e^{-\widehat{\lambda}(1-\theta)T} \leq C_4e^{-\lambda_p T}
		\end{equation*}
		for some $0< \lambda_p < \widehat{\lambda}(1-\theta)$ since $C_T$ grows at most polynomially in $T$, and similarly
		\begin{equation*}
		\|b_j(T) - b_{j,\infty}\|_p \leq \|b_j(T) - b_{j,\infty}\|_{\infty}^{\theta}\|b_j(T) - b_{j,\infty}\|_1^{1-\theta}\le C_5e^{-\lambda_p T}.
		\end{equation*}
	This concludes the proof of Theorem \ref{T1.3}.
	\end{proof}
\section{Entropy-entropy production Inequality}\label{sec:add}
	\begin{lemma}[Entropy-entropy production estimate]\label{EnEnDiEs}
		Let $\ainf \in (0,\infty)^{M}$ and $\binf\in (0,\infty)^N$ satisfy
		\begin{equation*}
		\ainf^{\alpha} = \binf^{\beta}
		\end{equation*}
		where $\alpha\in [1,\infty)^M$ and $\beta\in [1,\infty)^N$.
		
		Then, there exists an explicit constant $\lambda>0$ depending on $\ainf$, $\binf$, $\alpha$, $\beta$ and the domain $\Omega$, such that for any nonnegative functions $a = (a_i): \Omega\to \mathbb R_+^M$ and $b = (b_j): \Omega \to \mathbb R_+^N$ satisfying
		\begin{equation*}
		\beta_j\overline{a}_i + \alpha_i\overline{b}_j = \beta_ja_{i,\infty} + \alpha_ib_{j,\infty} \qquad \text{ for all }\quad i=1,\ldots, M, \; j=1,\ldots, N,
		\end{equation*}
		the following entropy-entropy production inequality holds
		\begin{equation*}
		\widetilde{D}[a,b] \geq \lambda(E[a,b] - E[\ainf,\binf])
		\end{equation*}
		where
		\begin{equation*}
		\widetilde{D}[a,b] = \sum_{i=1}^{M}\int_{\Omega}a_i\ln\frac{a_i}{\overline{a}_i}dx 
		+ \sum_{j=1}^{N}\int_{\Omega}b_j\ln\frac{b_j}{\overline{b}_j}dx 
		+ \int_{\Omega}(a^\alpha - b^\beta)\ln\frac{a^\alpha}{b^\beta}dx
		\end{equation*}
		and
		\begin{equation*}
		E[a,b] = \sum_{i=1}^{M}\int_{\Omega}(a_i\ln a_i - a_i + 1)dx + \sum_{j=1}^{N}\int_{\Omega}(b_j\ln b_j - b_j + 1)dx.
		\end{equation*}
	\end{lemma}
	\begin{remark}
		The above entropy-entropy production inequality was first proved in \cite{FT17a} in a constructive way with explicit bounds on the constant $\lambda$. 
%The proof was quite long due to the cases distinction depending on the smallness of $\overline{\sqrt{a_i}}$ and $\overline{\sqrt{b_j}}$. 
The proof stated here follows the line of a significantly simplified version presented in \cite{FT17}. 
	\end{remark}
	\begin{proof}
		First, by the additivity of the relative entropy, we have
		\begin{equation*}
		\begin{aligned}
		E[a,b] - E[\ainf,\binf] &= (E[a,b] - E[\overline{a}, \overline{b}]) + (E[\overline{a}, \overline{b}] - E[\ainf,\binf])\\
		&=\left[ \sum_{i=1}^{M}\int_{\Omega}a_i\ln\frac{a_i}{\overline{a}_i}dx + \sum_{j=1}^{N}\int_{\Omega}b_j\ln\frac{b_j}{\overline{b}_j}dx\right]\\
		&\quad+\left[\sum_{i=1}^{M}\left(\overline{a}_i\ln\frac{\overline{a}_i}{a_{i,\infty}} - \overline{a}_i + a_{i,\infty}\right) + \sum_{j=1}^{N}\left(\overline{b}_j\ln\frac{\overline{b}_j}{b_{j,\infty}} - \overline{b}_j + b_{j,\infty}\right) \right]\\
		&=: (I) + (II).
		\end{aligned}
		\end{equation*}
		It is straightforward that $(I)$ can be controlled by $\widetilde{D}[a,b]$, i.e. 
		\begin{equation*}
		\frac 12 \widetilde{D}[a,b] \geq \frac {1}{2}\times (I).
		\end{equation*}
		It remains to control $(II)$. To do that, we first introduce the following useful notations and definitions 
		\begin{equation*}
		A_i = \sqrt{a_i},\quad B_j = \sqrt{b_j}, \quad A_{i,\infty} = \sqrt{a_{i,\infty}}, \quad B_{j,\infty} = \sqrt{b_{j,\infty}},
		\end{equation*}
		\begin{equation*}
		\delta_i(x) = A_i(x) - \overline{A}_i, \qquad \eta_j(x) = B_j(x) - \overline{B}_j,
		\end{equation*}
		and
		\begin{equation*}
		A^{\alpha} = \prod_{i=1}^{M}A_i^{\alpha_i}, \quad B^{\beta} = \prod_{j=1}^{N}B_j^{\beta_j}.
		\end{equation*}
		By the elementary inequality $(x-y)\ln(x/y) \geq 4(\sqrt{x} - \sqrt{y})^2$, we have
		\begin{equation*}
		\int_{\Omega}a_i\ln\frac{a_i}{\overline{a}_i}dx = \int_{\Omega}\left(a_i\ln\frac{a_i}{\overline{a}_i} - a_i + \overline{a}_i\right)dx \geq 4\int_{\Omega}(\sqrt{a_i} - \sqrt{\overline{a}_i})^2dx \geq 4\|\delta_i\|^2
		\end{equation*}
		and similarly
%\begin{equation*}
$\int_{\Omega}b_j\ln\frac{b_j}{\overline{b}_j}dx \geq 4\|\eta_j\|^2$.
%\end{equation*}
Moreover,
%\begin{equation*}
$\int_{\Omega}(a^\alpha - b^\beta)\ln{\frac{a^\alpha}{b^\beta}}dx \geq 4\|A^\alpha - B^\beta\|^2$.
%\end{equation*}
Therefore,
		\begin{equation}\label{bao1}
		\frac 12 \widetilde{D}[a,b] \geq 2\left[\sum_{i=1}^{M}\|\delta_i\|^2 + \sum_{j=1}^{N}\|\eta_j\|^2 + \|A^{\alpha} - B^{\beta}\|^2\right].
		\end{equation}
In order to bound to estimate the right-hand-side of 
\eqref{bao1} with an upper bound of $(II)$, we first observe from the conservation laws
		\begin{equation*}
		\beta_j\overline{a}_i + \alpha_i\overline{b}_j = \beta_ja_{i,\infty} + \alpha_ib_{j,\infty}, \qquad \text{ for all } i, j.
		\end{equation*}
		that there exists a constant $M_0>0$ such that
		\begin{equation*}
		\overline{a}_i, \overline{b}_j \leq M_0^2, \qquad \text{ for all } i, j.
		\end{equation*}
		Next, we note that the two variables function 
		\begin{equation*}
		\Phi(x,y) = \frac{x\ln(x/y) - x + y}{(\sqrt x - \sqrt y)^2}
		\end{equation*}
		is continuous on $(0,\infty)^2$ and $\Phi(\cdot, y)$ is increasing for each fixed $y$. Then, the term $(II)$ is estimated as
		\begin{equation}\label{bao2}
		\begin{aligned}
		(II) &= \sum_{i=1}^{M}\Phi(\overline{a}_i,a_{i,\infty})(\sqrt{\overline{a}_i} - \sqrt{a_{i,\infty}})^2+ \sum_{j=1}^{N}\Phi(\overline{b}_j,b_{j,\infty})(\sqrt{\overline{b}_j} - \sqrt{b_{j,\infty}})^2\\
		&\leq \max_{i,j}\{\Phi(M_0^2,a_{i,\infty});\Phi(M_0^2, b_{j,\infty})\}\Biggl(\sum_{i=1}^M(\sqrt{\overline{A_i^2}} - {A_{i,\infty}})^2 + \sum_{j=1}^N(\sqrt{\overline{B_j^2}} - {B_{j,\infty}})^2\Biggr).
		\end{aligned}
		\end{equation}
		From \eqref{bao1} and \eqref{bao2}, it remains to show that
		\begin{equation}\label{bao3}
		\sum_{i=1}^{M}\|\delta_i\|^2 + \sum_{j=1}^{N}\|\eta_j\|^2 + \|A^{\alpha} - B^{\beta}\|^2 \geq C_0\Biggl(\sum_{i=1}^M(\sqrt{\overline{A_i^2}} - {A_{i,\infty}})^2 + \sum_{j=1}^N(\sqrt{\overline{B_j^2}} - {B_{j,\infty}})^2\Biggr)
		\end{equation}
		for some constant $C_0>0$. By using Lemma \ref{inter}, we have with $\overline{A} = (\overline{A}_1, \ldots, \overline{A}_M)$ and $\overline{B} = (\overline{B}_1, \ldots, \overline{B}_N)$
		\begin{equation}\label{bao4}
		\sum_{i=1}^{M}\|\delta_i\|^2 + \sum_{j=1}^{N}\|\eta_j\|^2 + \|A^{\alpha} - B^{\beta}\|^2 \geq C_1\Biggl(\sum_{i=1}^{M}\|\delta_i\|^2 + \sum_{j=1}^{N}\|\eta_j\|^2 + \left|\overline{A}^\alpha - \overline{B}^\beta\right|^2\Biggr)
		\end{equation}
		for some constant $C_1>0$. Using the ansatz
		\begin{equation}\label{ansatz}
		\overline{A_i^2} = A_{i,\infty}^2(1+\mu_i)^2 \quad \text{ and } \quad\overline{B_j^2} = B_{j,\infty}^2(1+\zeta_j)^2, \qquad \text{where} \quad \mu_i, \zeta_j \in [-1,\infty),
		\end{equation}
the right hand side of \eqref{bao3} writes as
\begin{equation}\label{bao5}
\text{RHS of (\ref{bao3})} = C_0\Biggl(\sum_{i=1}^M\mu_i^2 + \sum_{j=1}^N\zeta_j^2\Biggr).
\end{equation}
Moreover, the bounds $\overline{a_i}=\overline{A_i^2} \leq M_0^2$ and $\overline{b_j}=\overline{B_j^2} \leq M_0^2$ imply
		\begin{equation}
		-1 \leq \mu_i \leq M_1 \quad \text{ and } -1\leq \zeta_j \leq M_1
		\end{equation}
		for some constant $M_1>0$. From the ansatz \eqref{ansatz} (and similar to the proof of Lemma \ref{variant-eede}), we have
\begin{align*}
\overline{A}_i &= \sqrt{\overline{A_i^2}} - Q_i(A_i)\|\delta_i\| = A_{i,\infty}(1+\mu_i) - Q_i(A_i)\|\delta_i\|\\
\overline{B_j} &= \sqrt{\overline{B_j^2}} - R_j(B_j)\|\eta_j\|= B_{j,\infty}(1+\zeta_j) - R_j(B_j)\|\eta_j\|
\end{align*}
where 
\begin{equation*}
		0 \leq Q_i(A_i) := \frac{\|\delta_i\|}{\sqrt{\overline{A_i^2}} + \overline{A}_i} \leq 1 \quad \text{ and } \quad 0 \leq R_j(B_j):= \frac{\|\eta_j\|}{\sqrt{\overline{B_j^2}} + \overline{B}_j} \leq 1.
		\end{equation*}
Next, we use Taylor expansion to estimate
		\begin{equation*}
		\overline{A_i}^{\alpha_i} = \left(A_{i,\infty}(1+\mu_i)- Q_i(A_i)\|\delta_i\|\right)^{\alpha_i} = A_{i,\infty}^{\alpha_i}(1+\mu_i)^{\alpha_i} + {\widehat{Q}_i\|\delta_i\|}
		\end{equation*}
in which the Lagrange remainder term ${\widehat{Q}_i= \widehat{Q}(\mu_i,  \|\delta_i\|)}$ is uniformly bounded {above by a constant for all admissible values of $\mu_i$ and $\|\delta_i\|$}
thanks to the boundedness of $\mu_i$ and $\|\delta_i\|\leq \sqrt{\overline{A_i^2}} \leq M_0$. Similarly,
		\begin{equation*}
		\overline{B_j}^{\beta_j} = B_{j,\infty}^{\beta_j}(1+\zeta_j)^{\beta_j} + {\widehat{R}_j\|\eta_j\|}
		\end{equation*}
		with uniformly bounded remainder $ \widehat{R}_j(\zeta_j,\|\eta_j\|)$. Thus
		\begin{equation*}
		\begin{aligned}
		\Bigl|\overline{A}^\alpha - \overline{B}^\beta\Bigr|^2 &= \Biggl|\prod_{i=1}^M\overline{A}_i^{\alpha_i} - \prod_{j=1}^{N}\overline{B}_j^{\beta_j}\Biggr|^2\\
		&= \Biggl|\prod_{i=1}^M\left(A_{i,\infty}^{\alpha_i}(1+\mu_i)^{\alpha_i} + \widehat{Q}_i\|\delta_i\|\right) - \prod_{j=1}^{N}\left(B_{j,\infty}^{\beta_j}(1+\zeta_j)^{\beta_j} + \widehat{R}_j\|\eta_j\|\right)\Biggr|^2\\
		&= \Biggl|A_{\infty}^\alpha\prod_{i=1}^M(1+\mu_i)^{\alpha_i} - B_{\infty}^\beta\prod_{j=1}^{N}(1+\zeta_j)^{\beta_j} + \Theta(\widehat{Q}_i, \widehat{R}_j)\Biggl(\sum_{i=1}^M\|\delta_i\| + \sum_{j=1}^N\|\eta_j\|\Biggr)\Biggr|^2
		\end{aligned}
		\end{equation*}
		with $\Theta(\widehat{Q}_i, \widehat{R}_j)$ is also uniformly bounded. Thus, by using $(x+y)^2 \geq \frac 12 x^2 - y^2$ and $A_{\infty}^\alpha = \sqrt{\ainf^\alpha} = \sqrt{\binf^\beta} = B_{\infty}^\beta$ and the Cauchy-Schwarz inequality,
		\begin{equation}
		\left|\overline{A}^\alpha - \overline{B}^\beta\right|^2
		\geq \frac 12 A_{\infty}^\alpha\Biggl|\prod_{i=1}^M(1+\mu_i)^{\alpha_i} - \prod_{j=1}^{N}(1+\zeta_j)^{\beta_j}\Biggr|^2 - |\Theta|^2(M+N)^2\Biggl(\sum_{i=1}^M\|\delta_i\|^2 + \sum_{j=1}^N\|\eta_j\|^2\Biggr).
		\end{equation}
		Hence, for any $\delta\in (0,1)$  holds
		\begin{equation*}
		\begin{aligned}
		\sum_{i=1}^M\|\delta_i\|^2 &+ \sum_{j=1}^N\|\eta_j\|^2 + \left|\overline{A}^\alpha - \overline{B}^\beta\right|^2\\
		&\geq \sum_{i=1}^M\|\delta_i\|^2 + \sum_{j=1}^N\|\eta_j\|^2\\
		&\quad + \delta\Biggl(\frac 12 A_{\infty}^\alpha\Biggl|\prod_{i=1}^M(1+\mu_i)^{\alpha_i} - \prod_{j=1}^{N}(1+\zeta_j)^{\beta_j}\Biggr|^2 - |\Theta|^2(M+N)^2\Biggl(\sum_{i=1}^M\|\delta_i\|^2 + \sum_{j=1}^N\|\eta_j\|^2\Biggr)\Biggr)\\
		&\geq \frac{\delta}{2}A_{\infty}^\alpha\Biggl|\prod_{i=1}^M(1+\mu_i)^{\alpha_i} - \prod_{j=1}^{N}(1+\zeta_j)^{\beta_j}\Biggr|^2
		\end{aligned}
		\end{equation*}
		by choosing $\delta$ small enough such that $1 \geq \delta|\Theta|^2(M+N)^2$ since $\Theta$ is uniformly bounded above. This leads in combination with \eqref{bao4} to a lower bound of the left hand side of \eqref{bao3}
		\begin{equation}\label{bao6}
		\text{LHS of (\ref{bao3})} \geq C_1\frac{\delta}{2}A_{\infty}^\alpha\Biggl|\prod_{i=1}^M(1+\mu_i)^{\alpha_i} - \prod_{j=1}^{N}(1+\zeta_j)^{\beta_j}\Biggr|^2.
		\end{equation}
		From \eqref{bao5} and \eqref{bao6}, it is sufficient to prove
		\begin{equation}\label{bao7}
		\Biggl|\prod_{i=1}^M(1+\mu_i)^{\alpha_i} - \prod_{j=1}^{N}(1+\zeta_j)^{\beta_j}\Biggr|^2 \geq C_2\Biggl(\sum_{i=1}^M\mu_i^2 + \sum_{j=1}^N\zeta_j^2\Biggr).
		\end{equation}
		
		In order to do so, we note that the conservation laws
		\begin{equation*}
		\beta_j\overline{a}_i + \alpha_i\overline{b}_j = \beta_ja_{i,\infty} + \alpha_ib_{j,\infty}
		\end{equation*}
		rewritten in terms of the ansatz \eqref{ansatz}, i.e.
		\begin{equation*}
		\beta_jA_{i,\infty}^2(\mu_i^2 + 2\mu_i) + \alpha_iB_{j,\infty}^2(\zeta_j^2 + 2\zeta_j) = 0.
		\end{equation*}
		imply $\mu_i\zeta_j \leq 0$ thanks to $\mu_i, \zeta_j \geq -1$ for all $i, j$. Without loss of generality, we assume $\mu_i \geq 0$ and $\zeta_j \leq 0$ for all $i, j$. Then, for any $1\leq i_0 \leq M$ and $1\leq j_0 \leq N$,
		\begin{equation*}
		\begin{aligned}
		\Biggl|\prod_{i=1}^M(1+\mu_i)^{\alpha_i} - \prod_{j=1}^{N}(1+\zeta_j)^{\beta_j}\Biggr| &\geq \prod_{i=1}^M(1+\mu_i)^{\alpha_i} - \prod_{j=1}^{N}(1+\zeta_j)^{\beta_j}
		\geq (1+\mu_{i_0})^{\alpha_{i_0}} - (1+\zeta_{j_0})^{\beta_{j_0}}\\
		&\geq (1+\mu_{i_0}) - (1+\zeta_{j_0})
		\geq \mu_{i_0} - \zeta_{j_0} \geq 0.
		\end{aligned}
		\end{equation*}
Thus
		\begin{equation*}
		\Biggl|\prod_{i=1}^M(1+\mu_i)^{\alpha_i} - \prod_{j=1}^{N}(1+\zeta_j)^{\beta_j}\Biggr|^2 \geq (\mu_{i_0} - \zeta_{j_0})^2 = \mu_{i_0}^2 - 2\mu_{i_0}\zeta_{j_0} + \zeta_{j_0}^2 \geq \mu_{i_0}^2 + \zeta_{j_0}^2.
		\end{equation*}
		Since $1\leq i_0\leq M$ and $1 \leq j_0\leq N$ are arbitrary, we finally obtain \eqref{bao7} with $C_2 = 1/\max\{M; N\}$.
	\end{proof}
	\begin{lemma}\label{inter}
		Let $a_i, b_j$ be functions defined in Lemma \ref{EnEnDiEs}. Then, there exists a constant $C$ such that
		\begin{equation*}
		\sum_{i=1}^{M}\|\delta_i\|^2 + \sum_{j=1}^{N}\|\eta_j\|^2 + \|A^{\alpha} - B^{\beta}\|^2 \geq C\left|\overline{A}^\alpha - \overline{B}^\beta\right|^2.
		\end{equation*}
	\end{lemma}
	\begin{proof}
		Fix a constant $L>0$. Denote by
		\begin{equation*}
		S = \{x\in\Omega: |\delta_i(x)| \leq L, |\eta_j(x)|\leq L \text{ for all } i=1,\ldots, M, \; j=1,\ldots, N\} \quad \text{ and } \quad S^\perp = \Omega\backslash S.
		\end{equation*}
		Recalling $\overline{A_i} \leq \sqrt{\overline{A_i^2}} \leq M_0$ and $\overline{B_j} \leq \sqrt{\overline{B_j^2}}\leq M_0$, we use Taylor expansion to estimate
		\begin{equation}\label{bao8}
		\begin{aligned}
		\|A^\alpha - B^\beta\|^2 &\geq \int_{S}\biggl|\prod_{i=1}^{M}(\overline{A}_i + \delta_i(x))^{\alpha_i} - \prod_{j=1}^N(\overline{B}_j + \eta_j(x))^{\beta_j}\biggr|^2dx\\
		&\geq \frac 12 \Bigl|\overline{A}^\alpha - \overline{B}^\beta\Bigr|^2|S| - \widetilde{R}(\overline{A}_i, \overline{B}_j, |\delta_i|, |\eta_j|)\Biggl(\sum_{i=1}^{M}\|\delta_i\|^2 + \sum_{j=1}^N\|\eta_j\|^2\Biggr)
		\end{aligned}
		\end{equation}
		where $|\widetilde{R}| \leq C(M_0, L)$ due to the boundedness of $\delta_i$ and $\eta_j$ in $S$. In $S^\perp$, we have
		\begin{equation*}
		\sum_{i=1}^{M}\|\delta_i\|^2 + \sum_{j=1}^{N}\|\eta_j\|^2 \geq \int_{S^\perp}\Biggl(\sum_{i=1}^M|\delta_i(x)|^2 + \sum_{j=1}^N|\eta_j(x)|^2\Biggr)dx \geq L^2|S^\perp|.
		\end{equation*}
Next, there clearly exists a constant $\Lambda>0$ such that
%	\begin{equation*}
$\left|\overline{A}^\alpha - \overline{B}^\beta\right|^2 \leq \Lambda$
%\end{equation*}
since $\overline{A}_i, \overline{B}_j \leq M_0$. Therefore,
		\begin{equation}\label{bao9}
		\sum_{i=1}^{M}\|\delta_i\|^2 + \sum_{j=1}^{N}\|\eta_j\|^2 \geq L^2|S^\perp| \geq \frac{L^2}{\Lambda}\left|\overline{A}^\alpha - \overline{B}^\beta\right|^2|S^\perp|.
		\end{equation}
		Combining \eqref{bao8} and \eqref{bao9} we find for any $\theta_1, \theta_2 \in (0,1)$
		\begin{equation*}
		\begin{aligned}
		\sum_{i=1}^M\|\delta_i\|^2 + \sum_{j=1}^N\|\eta_j\|^2 + \|A^\alpha - B^\beta\|^2
		&\geq \theta_1\frac{L^2}{\Lambda}\left|\overline{A}^\alpha - \overline{B}^\beta\right|^2|S^\perp| + (1-\theta_1)\Biggl(\sum_{i=1}^M\|\delta_i\|^2 + \sum_{j=1}^N\|\eta_j\|^2\Biggr)\\
		&\quad + \theta_2\frac 12 \left|\overline{A}^\alpha - \overline{B}^\beta\right|^2|S| - \theta_2|\widetilde{R}|\Biggl(\sum_{i=1}^{M}\|\delta_i\|^2 + \sum_{j=1}^N\|\eta_j\|^2\Biggr)\\
		&\geq \min\left\{\theta_1\frac{L^2}{\Lambda}; \theta_2\frac 12 \right\}\left|\overline{A}^\alpha - \overline{B}^\beta\right|^2(|S| + |S^\perp|)\\
		&= \min\left\{\theta_1\frac{L^2}{\Lambda}; \theta_2\frac 12 \right\}\left|\overline{A}^\alpha - \overline{B}^\beta\right|^2
		\end{aligned}
		\end{equation*}
		by choosing $\theta_1, \theta_2$ small enough such that $1 - \theta_1 - \theta_2|\widetilde{R}| \geq 0$ and using $|S| + |S^\perp| = |\Omega| = 1$. The proof of Lemma \ref{inter} is hence complete.
	\end{proof}
	
\section{Proof Theorem \ref{thm:1.1}: existence of global weak solution to \eqref{S}}\label{app}
In this section, we give a proof Theorem \ref{thm:1.1} about the global existence of weak solutions to \eqref{S} under the conditions \eqref{G}-\eqref{M}-\eqref{P}.
Consider the approximating system
\begin{equation}\label{approx}
	\partial_t u_{i,\varepsilon} - d_i\Delta(u_{i,\varepsilon}^{m_i}) = f_{i,\varepsilon}(u_{\varepsilon}):= \frac{f_i(u_{\varepsilon})}{1+\varepsilon\sum_{i=1}^{S}|f_i(u_{\varepsilon})|}, \quad \nabla(u_{i,\varepsilon}^{m_i})\cdot \overrightarrow{n} = 0, \quad u_{i,\varepsilon}(x,0) = u_{i,0,\varepsilon}(x)
\end{equation}
where $u_{\varepsilon} = (u_{1,\varepsilon}, \ldots, u_{S,\varepsilon})$ and the sequence of approximating {nonnegative} initial data $u_{i,0,\varepsilon}\in L^{\infty}(\Omega)$ converges to $u_{i,0}$ in $L^2(\Omega)$. By the construction of the approximative system, it directly follows that the nonlinearities $f_{i,\varepsilon}$ still satisfy the conditions \eqref{M} and \eqref{P}. Moreover, for $\varepsilon>0$
\begin{equation*}
	|f_{i,\varepsilon}(u_{\varepsilon})| \leq \frac{|f_i(u_{\varepsilon})|}{1+\varepsilon\sum_{i=1}^{S}|f_i(u_{\varepsilon})|} \leq \frac{1}{\varepsilon} \quad \text{ for all } u_{\varepsilon}\in \mathbb R^S.
\end{equation*}
Hence, by a classical result for the porous medium equation with $L^{\infty}$ data, there exists a strong nonnegative solution $u_{\varepsilon} = (u_{i,\varepsilon})_{i=1\ldots S}$ (see e.g. \cite[Section 8]{Vaz07}) in the sense that
\begin{equation*}
	u_{i,\varepsilon}^{m_i}\in L^2_{loc}(0,+\infty;H^1(\Omega)), \quad \partial_tu_{i,\varepsilon} = d_i\Delta(u_{i,\varepsilon}^{m_i}) + f_{i,\varepsilon}(u_{\varepsilon}) \in L^1_{loc}(0,+\infty;L^1(\Omega)), 
\end{equation*} 
\begin{equation*}
	u_{i,\varepsilon}\in C([0,T); L^1(\Omega)) \text{ and } u_{i,\varepsilon}(0) = u_{i,0,\varepsilon},
\end{equation*}
and the equation for $u_{i,\varepsilon}$ holds a.e. in $Q_T$ for any $T>0$. Therefore, it follows immediately that
\begin{equation}\label{weakform}
	-\int_{\Omega}u_{i,0,\varepsilon}\psi(0)dx - \int_0^T\int_{\Omega}(\partial_t \psi u_{i,\varepsilon} + u_{i,\varepsilon}^{m_i}\Delta \psi)dxdt = \int_0^T\int_{\Omega}f_{i,\varepsilon}(u_{\varepsilon})\psi dxdt
\end{equation}
for any test function $\psi \in C^{2,1}(\overline{\Omega}\times [0,T])$ with $\psi(T) = 0$ and $\nabla \psi \cdot \overrightarrow{n} = 0$ on $\partial\Omega\times (0,T)$. 

In order to pass to the limit as $\varepsilon\to 0$ in the weak formula \eqref{weakform},
we use the following uniform a-priori estimates, which are a consequence of a duality argument 
in the spirit of e.g. \cite{Pie10} and references therein. 
\begin{lemma}[Duality estimates and uniform a-priori estimates for the approximating solutions, cf. \cite{LP17}]\label{duality}
		Let $u_{\varepsilon} = (u_{1,\varepsilon},\ldots, u_{S,\varepsilon})$ be the {nonnegative} solutions to the approximating system \eqref{approx}. Then,
		\begin{equation*}
		\|u_{i,\varepsilon}\|_{L^{m_i+1}(Q_T)} \leq C \quad \text{ for all } \quad T>0 \quad \text{ and } \quad i = 1,\ldots, S,
		\end{equation*}
		where the $\varepsilon$-independent constant $C$ only depends on the $L^2$-norm of the initial data, the positive constants $\lambda_i$ of assumption \eqref{M}, the positive diffusion coefficients $d_i$ and the domain $\Omega$.
		Moreover, we have 	
		\begin{equation*}
		\|f_{i,\varepsilon}(u_{\varepsilon})\|_{L^{1+\delta}(Q_T)} \leq C
	        \end{equation*}
for some $\delta > 0$, where the constant $C$ depends only on the $L^2$-norm of $u_{i,0,\varepsilon}$, the positive constants $\lambda_i$ of assumption \eqref{M}, the diffusion coefficients $d_i$, the exponents $m_i$ and the domain $\Omega$.
\end{lemma}

%\begin{lemma}[A priori estimates for approximating solutions]\label{UniBound}
%	The following estimates hold uniformly in $\varepsilon>0$, for all $i=1\ldots S$,
%	\begin{equation*}
%		\|u_{i,\varepsilon}\|_{L^{m_i+1}(Q_T)} \leq C,
%	\end{equation*}
%	\begin{equation*}
%		\|f_{i,\varepsilon}(u_{\varepsilon})\|_{L^{1+\delta}(Q_T)} \leq C
%	\end{equation*}
%	for some $\delta > 0$, where the constant $C$ depends only on the $L^2$-norm of $u_{i,0,\varepsilon}$, the domain $\Omega$, the diffusion coefficients $d_i$ and the exponents $m_i$.
%\end{lemma}
%\begin{proof}
%	Since the nonlinearities $f_{i,\varepsilon}(u_{\varepsilon})$ still satisfies the dissipation of mass condition \eqref{M}, we can apply Lemma \ref{duality} to obtain the estimate 
%		\begin{equation*}
%	\|u_{i,\varepsilon}\|_{L^{m_i+1}(Q_T)} \leq C,
%	\end{equation*}
%	where $C$ is independent of $\varepsilon$. For the estimate concerning the nonlinearities we have
%	\begin{equation*}
%		|f_{i,\varepsilon}(u_{\varepsilon})| \leq |f_{i}(u_{\varepsilon})| \leq C(1+|u_{\varepsilon}|^{\nu})
%	\end{equation*}
%	with $C$ does not depend on $\varepsilon$. Using now the condition $m_{i} > \nu - 1$ and the estimate of $\|u_{i,\varepsilon}\|_{L^{m_i+1}(Q_T)}$ we obtain the desired estimate for $f_{i,\varepsilon}(u_{\varepsilon})$.
%\end{proof}

	\begin{proof}
		The proof follows \cite{LP17} with straightforward changes due to the considered Neumann (instead of Dirichlet) boundary conditions. By setting	
		\begin{equation*}
		Z = \sum_{i=1}^{S}\lambda_iu_{i,\varepsilon}\quad \text{ and } \quad W = \sum_{i=1}^{S}d_i\lambda_iu_{i,\varepsilon}^{m_i}
		\end{equation*}
		and by summing up the equations of systems \eqref{S}, the mass dissipation property \eqref{M} implies
		\begin{equation*}
		\partial_t Z - \Delta W \leq 0 \quad \text{ and } \quad \nabla W \cdot \overrightarrow{n} = 0.
		\end{equation*}
		Then, integration over $(0,t)$ and multiplication with $W(t)$ in $L^2(\Omega)$ (due to the regularity of the approximative solutions) leads 
		after integration over $\Omega$ to
		\begin{equation}\label{f1}
		\int_{\Omega}\left(Z(t)-Z(0)\right)W(t)dx - \int_{\Omega} W(t)\Delta\!\int_0^tW(s)ds dx \leq 0.
		\end{equation}
		Next, we integrate by parts with homogeneous Neumann boundary conditions the second term on the left hand side and calculate 
		%By putting $\varphi(t) = \int_0^tW(s)ds$ we have
		\begin{equation*}
		- \int_{\Omega} W(t) \Delta\!\int_0^tW(s)ds\; dx = \int_{\Omega}\nabla W(t)\cdot \nabla\!\int_0^tW(s)ds\;dx %= \int_{\Omega}\nabla \varphi'(t)\cdot\nabla \varphi(t)dx 
		= \frac{1}{2}\frac{d}{dt}\int_{\Omega}|\nabla\!  \int_0^tW(s)ds|^2dx.
		\end{equation*}
		Therefore, by integrating \eqref{f1} with respect to $t$ on $(0,T)$, we obtain
		\begin{equation}\label{f2}
		\int_{0}^{T}\!\!\int_{\Omega}Z(t)W(t)dxdt + \frac{1}{2}\int_{\Omega}|\nabla \!  \int_0^TW(s)ds|^2dx \leq \int_{0}^{T}\!\!\int_{\Omega}Z(0)W(t)dxdt.
		\end{equation}
		Moreover, we note that
		\begin{equation}\label{f3}
		%\begin{aligned}
		\int_0^T\!\!\int_{\Omega}Z(t)W(t)dxdt = \int_0^T\!\!\int_{\Omega}\left(\sum_{i=1}^{S}\lambda_iu_i\right)\left(\sum_{i=1}^{S}d_i\lambda_iu_i^{m_i}\right)dxdt
		\geq \sum_{i=1}^{S}d_i\lambda^2_i\|u_{i}\|_{L^{m_i+1}(Q_T)}^{m_i+1}
		%\end{aligned}
		\end{equation}
		due to the nonnegativity of functions $u_i$ and the constant $\lambda_i$. 
		To estimate the right hand side of \eqref{f2} in terms of the $L^2$-norm of $Z(0)$, we first notice from $\partial_t Z - \Delta W \leq 0$ that
		\begin{equation*}
			Z(T) - \Delta \int_0^TWdt \leq Z(0).
		\end{equation*}
		Multiplying this inequality with $\theta_0$ in $L^2(\Omega)$, where $\theta_0 \geq 0$ solves $-\Delta \theta_0 = Z(0)$, $\theta_0|_{\partial \Omega} = 0$, and using integration by parts $-\int_{\Omega}\theta_0\Delta\int_0^TW(t)dtdx = -\int_{\Omega}\Delta \theta_0\int_0^TW(t)dtdx$, leads to
		\begin{equation*}
			\int_{\Omega}Z(T)\theta_0 dx + \int_{\Omega}\left(Z(0)\int_0^TW(t)dt\right)dx \leq \int_{\Omega}Z(0)\theta_0dx = \|\nabla \theta_0\|^2 \leq C\|Z(0)\|^2,
		\end{equation*}
		which, together with $\int_{\Omega}Z(T)\theta_0dx \geq 0$, implies
		\begin{equation}\label{f4}
			\int_0^T\!\!\int_{\Omega}Z(0)W(t)dxdt \leq C\|Z(0)\|^2.
		\end{equation}
		By inserting \eqref{f3} and \eqref{f4} into \eqref{f2}, we obtain 
		\begin{equation*}
		\sum_{i=1}^{S}d_i\lambda^2_i\|u_{i,\varepsilon}\|_{L^{m_i+1}(Q_T)}^{m_i+1} %+ \frac{1}{2}\int_{\Omega}| \nabla\!  \int_0^TW(s)ds)|^2dx 
		\leq C\|Z(0)\|^{2},
		\end{equation*}
		which completes the proof of the first a-priori estimate of Lemma \ref{duality}.
		
		Concerning the second uniform a-priori estimate for the nonlinearities, we have
	\begin{equation*}
		|f_{i,\varepsilon}(u_{\varepsilon})| \leq |f_{i}(u_{\varepsilon})| \leq C(1+|u_{\varepsilon}|^{\nu}),
	\end{equation*}
	where $C$ does not depend on $\varepsilon$. By the assumption $m_{i} > \nu - 1$ and the estimate of $\|u_{i,\varepsilon}\|_{L^{m_i+1}(Q_T)}$, we obtain $\|f_{i,\varepsilon}(u_{\varepsilon})\|_{L^{1+\delta}(Q_T)} \leq C$.
	\end{proof}

The following compactness lemma allows to extract a converging subsequence from the approximating system.
\begin{lemma}\cite{Bar78}\label{compact}
	Let $m> (d-2)_+/d$ with $(d-2)_+ = \max\{0,d-2\}$. The mapping $L^1(\Omega)\times L^1(Q_T)\ni (u_0, f) \mapsto u \in L^1(Q_T)$ where $u\in C([0,T];L^1(\Omega))$ is the weak solution to 
	\begin{equation*}
		\partial_t u - \delta\Delta(u^{m}) = f, \quad \nabla(u^m)\cdot \overrightarrow{n} = 0, \quad u(0) = u_0,
	\end{equation*}
	with $\delta>0$,
	is compact.
\end{lemma}
\begin{proof}[Proof of Theorem \ref{thm:1.1}]
	Thanks to the uniform bounds of the nonlinearities in Lemma \ref{duality} and the compactness Lemma \ref{compact}, there exists a subsequence (not relabeled) $\{u_{i,\varepsilon}\}_{\varepsilon}$ which converges in $L^1(Q_T)$ to limit functions $u_{i}\in L^1(Q_T)$. From the $L^{m_i+1}$-bound in Lemma \ref{duality}, it holds in fact that $u_{i,\varepsilon}$ (up to another subsequence) converges strongly to $u_{i}$ in $L^{m_i}(Q_T)$. For the nonlinearities, we first notice from Lemma \ref{duality} that the sequence $\{f_{i,\varepsilon}(u_{\varepsilon})\}$ is uniformly integrable. Moreover, for another subsequence $u_{i,\varepsilon} \to u_i$ a.e. in $Q_T$ it follows that
	\begin{equation*}
		f_{i,\varepsilon}(u_{\varepsilon}) \to f_i(u_i) \quad \text{ a.e. in } \quad Q_T.
	\end{equation*}
	Therefore, we can apply Vitali's Lemma, see e.g. \cite[Chapter 16]{Sch05}, to obtain $f_{i,\varepsilon}(u_{\varepsilon}) \to f_{i}(u_i)$ strongly in $L^1(Q_T)$. All this allows to pass to the limit in the weak formulation \eqref{weakform} for any test function $\psi \in C^{2,1}(\overline{\Omega}\times [0,T])$ with $\psi(T) = 0$ and $\nabla \psi \cdot \overrightarrow{n} = 0$ on $\partial\Omega\times (0,T)$. Hence, we get
	\begin{equation*}
	-\int_{\Omega}\psi(0)u_{i,0}dx - \int_{Q_T}(\partial_t \psi u_{i} + u_{i}^{m_i}\Delta \psi)dxdt = \int_{Q_T}f_{i}(u)\psi dxdt.
	\end{equation*}
	The additional regularity $u^{m_i}_i \in L^1(0,T;W^{1,1}(\Omega))$ follows immediately from \cite[Lemma 4.7]{Luk10}, where % to have that
	\begin{equation*}
		\int_0^T\int_{\Omega}|\nabla u_i^{m_i}|^{\beta}dxdt \leq C(T, \|u_{i,0}\|_1, \|f_i(u)\|_{L^1(Q_T)}) \quad \text{ for all } 1\leq \beta < 1 + \frac{1}{1+m_id}.
	\end{equation*}
	From the above estimate and $f_i(u)\in L^1(Q_T)$, we also have $\partial_tu_i \in L^1(0,T;(W^{1,1}(\Omega))^*)$ which implies in particular $u_i \in C([0,T];L^1(\Omega))$. This completes the proof of existence of global weak solutions.
\end{proof}
\noindent{\bf Acknowledgements.} 
The second author was supported by the DFG Project CH 955/3-1.
This work is partially supported by International Research Training Group IGDK 1754 and NAWI Graz.

\end{document}